\documentclass[11pt]{amsart}
\usepackage{amssymb,enumitem}
\usepackage{amsthm}
\usepackage{tikz,graphicx}
\usepackage{xcolor}
\usetikzlibrary{arrows,positioning,shapes}
\usepackage[mathscr]{euscript}
\usepackage[margin=.89in]{geometry}
\usepackage{subcaption}
\usepackage{blkarray}

\usepackage[colorlinks,citecolor=cyan,linkcolor=purple]{hyperref}

\usepackage{cleveref}

\allowdisplaybreaks

\newtheorem{corollary}{Corollary}[subsection]
\newtheorem{theorem}[corollary]{Theorem}
\newtheorem{lemma}[corollary]{Lemma}
\newtheorem{prop}[corollary]{Proposition}
\theoremstyle{definition}
\newtheorem{defn}[corollary]{Definition}
\newtheorem{example}[corollary]{Example}
\newtheorem{remark}[corollary]{Remark}

\AddToHook{env/theorem/begin}{\crefalias{corollary}{theorem}}
\AddToHook{env/lemma/begin}{\crefalias{corollary}{lemma}}
\AddToHook{env/prop/begin}{\crefalias{corollary}{prop}}
\AddToHook{env/defn/begin}{\crefalias{corollary}{defn}}
\AddToHook{env/example/begin}{\crefalias{corollary}{example}}
\AddToHook{env/remark/begin}{\crefalias{corollary}{remark}}

\numberwithin{equation}{subsection}

\newcommand{\st}{\colon}
\newcommand{\into}{\hookrightarrow}		
\newcommand{\ol}[1]{\overline{#1}}

\makeatletter
\newcommand{\mylabel}[2]{(#2)\def\@currentlabel{#2}\label{#1}}
\makeatother

\DeclareMathOperator{\at}{at}
\DeclareMathOperator{\sgn}{sgn}
\DeclareMathOperator{\Hom}{Hom}
\DeclareMathOperator{\lcm}{lcm}

\newcommand{\join}{\vee}
\newcommand{\meet}{\wedge}
\newcommand{\I}{\mathcal{I}}
\renewcommand{\P}{\mathcal{P}}
\newcommand{\rk}{\rho}
\newcommand{\F}{\mathcal{F}}
\newcommand{\Q}{\mathbb{Q}}
\newcommand{\C}{\mathbb{C}}
\newcommand{\R}{\mathbb{R}}

\newcommand{\Z}{\mathbb{Z}}
\newcommand{\A}{\mathcal{A}}
\newcommand{\nbc}{\textsc{nbc}}
\newcommand{\G}{\mathbb{G}}
\newcommand{\M}{\mathcal{M}}
\newcommand{\zero}{\hat{0}}
\newcommand{\mo}{<_{\mathrm{grlex}}}
\newcommand{\ato}{<_{\mathrm{at}}}
\newcommand{\io}{\, {}_* \!\! >}
\newcommand{\B}{\mathcal{B}}

\DeclareMathOperator{\lt}{\textsc{LT}}
\DeclareMathOperator{\lm}{\textsc{LM}}
\DeclareMathOperator{\inv}{inv}

\title[The Orlik--Solomon algebra and cohomology of complex abelian arrangements]{The Orlik--Solomon algebra of a locally geometric poset and cohomology of complex abelian arrangements}

\author{Christin Bibby}
\address{Louisiana State University, Baton Rouge, LA, USA}
\email{\url{bibby@math.lsu.edu}}

\author{Peter Ramsey}
\address{Louisiana State University, Baton Rouge, LA, USA}
\email{\url{pramse3@lsu.edu}}

\keywords{matroid, geometric lattice, abelian arrangement, Orlik--Solomon algebra, graded-commutative algebra, Gr\"obner basis, Leray spectral sequence}

\subjclass[2020]{
Primary
05E16; 
Secondary
05B35, 
13P10, 
14N20, 
52C35, 
55N30, 
55T99
}

\thanks{C.B. was supported by NSF DMS-2204299. P.R. was supported by NSF DMS-2231492}

\begin{document}

\begin{abstract}
We construct an Orlik--Solomon algebra for any locally geometric poset as a natural generalization of the one for geometric lattices. This algebra has several interesting features, including a combinatorial no-broken-circuit  
basis that we obtain through Gr\"obner basis theory. When the poset captures the intersection data of an arrangement of certain subgroups in a complex abelian Lie group, we infuse the Orlik--Solomon algebra with topological information to compute the cohomology of the arrangement complement.
In particular, when the Lie group is compact, we present an explicit differential graded algebra whose cohomology is the rational cohomology of the arrangement complement.
We further exhibit a family of abelian arrangements whose integer cohomology has torsion.
\end{abstract}

\maketitle

\vspace{-3mm}

\tableofcontents

\vspace{-3mm}

\section{Introduction}

Due to groundbreaking work of Brieskorn \cite{brieskorn} and Orlik and Solomon \cite{OS}, the cohomology ring of the complement to a union of hyperplanes in a complex vector space has  a purely combinatorial presentation.
This means that the algebra does not depend on the linear equations defining the hyperplanes, but only on the \emph{geometric lattice} or \emph{matroid} encoding 
their intersection data.
There is thus an Orlik--Solomon algebra associated to any geometric lattice, regardless of realizability by a hyperplane arrangement.
In this combinatorial context, Orlik and Solomon's algebra enjoys several interesting features corresponding to Brieskorn's topological discoveries in the cohomology ring. 
For instance, it admits a grading by the geometric lattice and a Hilbert series formula in terms of the lattice's characteristic polynomial. Later, Jambu and Terao \cite{JT} determined a  combinatorial vector space basis in terms of the \emph{no-broken-circuits} or $\nbc$ sets, which was observed by Yuzvinsky \cite{yuz} to correspond to a Gr\"obner basis in the defining ideal.

Motivated by the intersection pattern of an \emph{abelian arrangement}, or a collection of certain subgroups in an abelian Lie group (\Cref{def:arrangement}), 
we generalize Orlik and Solomon's construction to assign a combinatorial algebra $A(\P)$ to any \emph{locally geometric poset}  $\P$ (see Definitions \ref{defn:geometric} and \ref{defn:OS}). 
The main idea of the ``locally geometric'' property is that closed intervals in the poset are geometric lattices, and this corresponds to how the subgroups in an abelian arrangement \emph{locally intersect like hyperplanes}. 
We thus take a local-to-global approach:
the Orlik--Solomon algebras associated to the local geometric lattices together form a flasque sheaf, and our algebra $A(\P)$ arises as global sections of this Orlik--Solomon sheaf (\Cref{thm:sheaf}).
This perspective was inspired by a related sheaf used in \cite{yuz-sheaf} for geometric semilattices.

We further prove that the algebra $A(\P)$ shares many of the same interesting features as in the case of geometric lattices.
Notably, we extend the notion of $\nbc$ to this context (\Cref{defn:nbc}) and establish a combinatorial $\nbc$ basis for $A(\P)$ using Gr\"obner basis theory (see \Cref{thm:gb,cor:vsbasis}).
This is then used to establish a $\P$-decomposition of $A(\P)$  and a Hilbert series formula in terms of the characteristic polynomial of $\P$ (\Cref{cor:brieskorn}), as well as the recursive tool of a deletion-contraction short exact sequence (\Cref{thm:del-con}).
With the Gr\"obner basis as the driving force here, our methods differ from those used for geometric lattices in \cite{OS,JT}. 

When the locally geometric poset $\P$ captures the intersection data of a complex abelian arrangement $\A$, the cohomology of the arrangement complement can then be computed using both the combinatorial data (in the form of $A(\P)$) and the topological data (from the cohomology of the ambient space).  
To this end, we define a \emph{topological Orlik--Solomon algebra} $B(\A)$ (\Cref{def:B}) and prove that it arises as the second page in a Leray spectral sequence (\Cref{thm:Leray}).
When the Lie group is noncompact, we use the cohomology calculation of \cite{BPP} to prove that in most cases,
the algebra $B(\A)$ is isomorphic to the integer cohomology of the arrangement complement (\Cref{thm:noncpt}).
There is, unfortunately, a family of noncompact Lie groups (including $\C^\times$) where the cup product cannot be recovered from the spectral sequence (see \Cref{rmk:noncpt}), and in this case we obtain an associated graded algebra of the integer cohomology ring (\Cref{thm:noncpt}).

When the Lie group is compact, the spectral sequence has a nontrivial differential separating $B(\A)$ from the cohomology.
In this case, we provide an explicit differential graded algebra (dga) $(B(\A)\otimes\Q,\delta)$ whose cohomology is the rational cohomology ring of the arrangement complement (\Cref{thm:cpt}). 
This greatly generalizes the dga model for  elliptic arrangements from \cite{bibby-elliptic} in two ways: managing more complicated combinatorics and allowing 
higher-dimensional Lie groups.
The techniques used to prove that $(B(\A)\otimes \Q;\delta)$ is a dga model rely heavily on rational coefficients: we analyze the mixed Hodge structure on a Leray spectral sequence.
To extract information about integer cohomology from the spectral sequence, there may be higher differentials or extensions to contend with.
Interestingly, and contrary to the noncompact case, there can be torsion in the integer cohomology of an abelian arrangement when the Lie group is compact (\Cref{ex:H(B;Q)}).

\medskip

Preliminary results of this paper were announced in the extended abstract \cite{BR} in the conference proceedings of Formal Power Series and Algebraic Combinatorics (FPSAC).

\noindent\textbf{Outline.} 
\Cref{sec:comb} sets the combinatorial foundations needed. We review poset terminology and define locally geometric posets in \Cref{sec:lgp}, then  associated matroidal notions in \Cref{sec:matroid}. 
We introduce the notion of unicircular elements in \Cref{sec:uni} and $\nbc$ elements in \Cref{sec:nbc}, proving several technical lemmas to be used later.
\Cref{sec:del-con} focuses on the deletion and contraction operations.

\Cref{sec:OS} is centered around the Orlik--Solomon algebra. We define the algebra in \Cref{sec:OSdef}, providing also an alternative presentation and comparing with the geometric lattice case.
In \Cref{sec:gb}, we review the theory of Gr\"obner bases in graded-commutative algebras and determine a Gr\"obner basis for the Orlik--Solomon ideal.
The $\nbc$ basis, $\P$-decomposition, and Hilbert series are proven in \Cref{sec:nbcbasis}. 
In \Cref{sec:sheaf}, we define the Orlik--Solomon sheaf and compute global sections.

\Cref{sec:H^*} focuses on the cohomology of complex abelian arrangements.
Abelian arrangements are defined in \Cref{sec:arrangement}, and the topological Orlik--Solomon algebra is defined in \Cref{sec:topOS}. 
We relate the topological Orlik--Solomon algebra to the Leray spectral sequence in \Cref{sec:Leray}, then compute cohomology in \Cref{sec:noncpt} for the noncompact case and in \Cref{sec:cpt} for the compact case.

\section{Combinatorial foundations}\label{sec:comb}
\subsection{Locally geometric posets}\label{sec:lgp}

Let $\P$ be a finite partially ordered set, or poset. 
Given $x\in\P$ we write 
\[\P_{\leq x}:=\{y\in\P\colon y\leq x\}\quad\textup{and}\quad\P_{\geq x}:=\{y\in\P\colon y\geq x\}.\]
For a subset $T\subseteq\P$, let the $\textbf{join}$, $\bigvee T$, be the set of minimal upper bounds, and the $\textbf{meet}$, $\bigwedge T$, be the set of maximal lower bounds. That is,
\[\bigvee T:=\min\{u\in\P\colon u\geq a,\;\forall a\in T\}\quad\textup{and}\quad\bigwedge T:=\max\{l\in\P\colon l\leq a,\;\forall a\in T\}.\]
When $T=\{x,y\}$ we write $x\join y:=\bigvee T$ and $x\meet y:=\bigwedge T$. We say that $\P$ is a \textbf{(meet-)semilattice} if for all $x,y\in\P$, $|x\meet y|=1$. If additionally $|x\join y|=1$ for all $x,y\in\P$, we say that $\P$ is a \textbf{lattice}. 
Say that $\P$ is \textbf{locally-lattice} if $\P_{\leq x}$ is a lattice for all $x\in\P$.

\begin{example}
The three posets whose Hasse diagrams are depicted in \Cref{fig:ex:P} are locally-lattice posets. None are lattices, although \Cref{fig:P1} is a semilattice.
In the poset of \Cref{fig:P2}, observe that 
$\bigvee\{1,2,4\}=c\vee 4=\{u_1,u_2\}$ and $u_1\meet u_2=\{c,x,y,z\}$.

\begin{figure}[hbt]
\begin{subfigure}[t]{.3\textwidth}
\centering
\begin{tikzpicture}
\tikzstyle{every node}=[draw,circle,fill=black,minimum size=4pt,inner sep=0pt]
\node (0) at (0,0) {};
\node (1) at (-2,1) {};
\node (2) at (-1,1) {};
\node (3) at (0,1) {};
\node (4) at (1,1) {};
\node (5) at (2,1) {};
\node (a) at (-1,2) {};
\node (b) at (2,2) {};
\node (c) at (1,2) {};
\foreach \x in {1,2,3,4,5} {\draw[-] (0)--(\x);};
\foreach \x in {1,2,3} {\draw[-] (\x)--(a);};
\foreach \x in {4,5} {\draw[-] (b)--(\x)--(c);};
\end{tikzpicture}
\caption{A locally geometric poset}
\label{fig:lgp}
\end{subfigure}
\begin{subfigure}[t]{.3\textwidth}
\centering
\begin{tikzpicture}
\node (0) at (0,0) {\scriptsize$\hat{0}$};
\node (1) at (-1.5,1) {\scriptsize $1$};
\node (2) at (-.5,1) {\scriptsize $2$};
\node (3) at (.5,1) {\scriptsize $3$};
\node (4) at (1.5,1) {\scriptsize $4$};
\node (c) at (-1.5,2) {\scriptsize $c$};
\node (x) at (-.5,2) {\scriptsize $x$};
\node (y) at (.5,2) {\scriptsize $y$};
\node (z) at (1.5,2) {\scriptsize $z$};
\node (u1) at (-1,3) {\scriptsize $u_1$};
\node (u2) at (1,3) {\scriptsize $u_2$};
\foreach \x in {1,2,3,4} {\draw[-] (0.north)--(\x.south);};
\foreach \x in {1,2,3} {\draw[-] (\x.north)--(c.south);};
\foreach \x in {c,x,y,z} {\draw[-] (u1.south)--(\x.north)--(u2.south);};
\draw[-] (1.north)--(x.south)--(4.north);
\draw[-] (2.north)--(y.south)--(4.north);
\draw[-] (3.north)--(z.south)--(4.north);
\end{tikzpicture}
\caption{A geometric poset}
\label{fig:P2}
\end{subfigure}
\begin{subfigure}[t]{.33\textwidth}
\centering
\begin{tikzpicture}[scale=1.5]
\node (0) at (0,0) {\scriptsize$\hat{0}$};
\node (1) at (-1.5,1) {\scriptsize $1$};
\node (2) at (-.5,1) {\scriptsize $2$};
\node (3) at (.5,1) {\scriptsize $3$};
\node (4) at (1.5,1) {\scriptsize $4$};
\node (p) at (-1,2) {\scriptsize $p$};
\node (q) at (0,2) {\scriptsize $q$};
\node (r) at (1,2) {\scriptsize $r$};
\foreach \x in {1,2,3,4} {\draw[-] (0.north)--(\x.south);};
\foreach \x in {1,2,3} {\draw[-] (\x.north)--(p.south);};
\draw[-] (2.north)--(q.south)--(4.north)--(r.south)--(3.north);
\end{tikzpicture}
\caption{A geometric semilattice}
\label{fig:P1}
\end{subfigure}
\caption{Examples of locally geometric posets}
\label{fig:ex:P}
\end{figure}
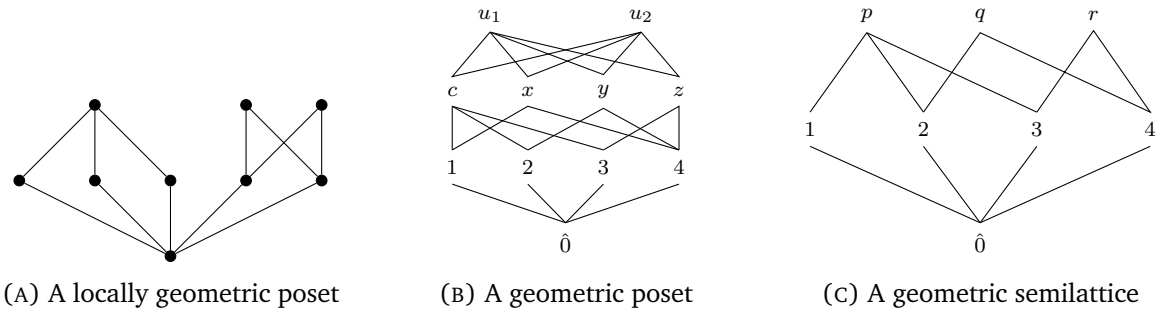

\end{example}

Before moving on, we state a technical lemma about joins that will be needed later.
\begin{lemma}[Associativity of join]\label{lem:triplejoin}
Let $\P$ be a locally-lattice poset. Then, for any $T\subseteq\P$ and $x\in\P$,
\[ 
\bigvee(T\cup\{x\}) = \bigcup_{y\in\vee T} (y\vee x).
\]
\end{lemma}
\begin{proof}
Suppose $u\in y\vee x$ for some $y\in\bigvee T$. 
Since $u$ lies above every element of $T\cup\{x\}$, there is some $v\in\bigvee(T\cup\{x\})$ with $v\leq u$. 
Now $v$ is a common upper bound of $T$, thus $v\geq y'$ for some $y'\in\bigvee T$.
Since $\P_{\leq u}$ is a lattice, the only element of $\bigvee T$ which lies below $u$ is $y$, hence $y=y'$. 
Then minimality of $u$ among common upper bounds of $y$ and $x$ implies $u=v\in\bigvee(T\cup\{x\})$.

Conversely, suppose $u\in\bigvee(T\cup\{x\})$. 
Since $u$ lies above every element of $T$, there is some $y\in\bigvee T$ such that $u\geq y$.
Now $u$ is a common upper bound of $x$ and $y$, thus $u\geq v$ for some $v\in x\vee y$. 
As was just proven, this implies $v\in \bigvee (T\cup \{x\})$ and, by minimality, $u=v\in\bigcup_{y\in \vee T} (y\vee x)$.
\end{proof}

Continuing to establish terminology,
a poset $\P$ is \textbf{bounded below} if it has a unique minimum element denoted $\hat{0}$. Say that $\P$ is \textbf{ranked} if there is an order-preserving map $\rk:\P\to\mathbb{Z}_{\geq 0}$, such that whenever $y$ covers $x$, $\rk(y)=\rk(x)+1$.
When a ranked poset $\P$ is bounded below, we may assume that $\rho(\hat{0})=0$, let $\rk(\P)=\max\{\rk(x)\st x\in\P\}$ be the maximum rank, and let the set of \textbf{atoms} of $\P$ be
\[\at(\P):=\{a\in\P:\rho(a)=1\}.\]
Abbreviate $\at(x):=\at(\P_{\leq x})$ whenever $x\in\P$.
When every element of $\P$ is contained in the join of some set of atoms, $\P$ is called \textbf{atomic}.

\begin{defn}\label{defn:geometric}
A ranked, atomic lattice $\P$ with rank function $\rho$ is  a \textbf{geometric lattice} if $\P$ is \textit{semimodular}, that is, for every $x,y\in\P$, 
    $\rk(x)+\rk(y)\geq\rho(x\join y)+\rho(x\meet y).$
    A ranked, bounded-below poset $\P$ is called \textbf{locally geometric} if $\P_{\leq x}$ is a geometric lattice for all $x\in\P$.
    A \textbf{geometric poset} is a locally geometric poset that satisfies the global condition:
whenever $x\in\P$, $Y\subseteq\at(\P)$,  $y\in\bigvee Y$, and $\rk(x)<\rk(y)=|Y|$, there exists an $a\in Y$ such that $a\not\leq x$ and $a\join x\neq\varnothing$. 
    \end{defn}

\begin{example}\label{ex:P}
The posets whose Hasse diagrams are depicted in \Cref{fig:ex:P} are locally geometric posets. 
While \Cref{fig:P1,fig:P2} satisfy the global geometric condition, \Cref{fig:lgp} does not.
\end{example}

There are several well-known classes of locally geometric posets.
For a finite set $T$, the \textbf{Boolean lattice} $2^T$, whose elements are the subsets of $T$ ordered by inclusion, is a geometric lattice. In this case, join corresponds to union; meet corresponds to intersection; and rank corresponds to cardinality. 
A \textbf{simplicial poset} is a bounded-below poset $S$ in which $S_{\leq \alpha}\cong 2^{\at(\alpha)}$ for all $\alpha\in S$. 
As such, a simplicial poset is a locally geometric poset. 
A \textbf{geometric semilattice} is a geometric poset that is also a meet-semilattice; this is equivalent to the definition of geometric semilattice from \cite{WW}.

The primary motivation for locally geometric posets is to capture the intersection data of an arrangement of submanifolds which \emph{intersect like hyperplanes}. This includes hyperplane arrangements, toric arrangements, and more generally abelian arrangements, which are the focus of \Cref{sec:H^*} and defined in \Cref{def:arrangement} (see \Cref{ex:arrangement} and \Cref{thm:P(A)geom}).
It also includes blowups of hyperplane arrangements \cite[Prop 2.3.6]{BDF} which arise in wonderful compactifications of hyperplane arrangement complements.

We will most often consider locally geometric posets, \emph{without} requiring the stronger condition of a geometric poset. However, the posets arising from abelian arrangements in \Cref{sec:H^*} happen to be geometric, and there is one place (see \Cref{rmk:geom-joinindep}) where the global geometric condition is needed.

\subsection{Matroid preschemes}\label{sec:matroid}

While the definitions in this section and the next are primarily pulled from \cite{bibby}, we redirect emphasis for the sake of this work. 
In particular, it will be most useful for us to define matroid preschemes in terms of independence, as in \cite[Thm. 6.2]{bibby}.

The ideas here use simplicial posets, for which additional notation is needed. 
Let $S$ be a simplicial poset. Use $|\alpha|:=|\at(\alpha)|$ to denote the rank of an element $\alpha\in S$. Furthermore, for $\alpha,\beta\in S$, if $\alpha\leq \beta$ then there is a unique element $\beta\setminus \alpha\in S$ such that $(\beta\setminus \alpha)\vee \alpha\ni \beta$ and $(\beta\setminus \alpha)\wedge \alpha=\{\zero\}$. Call this element the complement of $\alpha$ in $S_{\leq\beta}$, and note it satisfies $|\beta\setminus \alpha|=|\beta|-|\alpha|$.

\begin{defn}\label{defn:matroids}
A \textbf{matroid prescheme} is a pair $M=(S,\I)$, where $S$ is a finite simplicial poset and $\I\subseteq S$ is a subset of \textbf{independent} elements satisfying:
\begin{enumerate}[label=\textbf{(I\arabic*)}, ref=\textbf{(I\arabic*)}]
\item[\mylabel{I1}{\textbf{I1}}] $\I$ is nonempty.
\item[\mylabel{I2}{\textbf{I2}}] If $\alpha,\beta\in S$ with $\alpha\leq \beta$ and $\beta\in\I$, then $\alpha\in \I$.
\item[\mylabel{I3}{\textbf{I3}}] If $\alpha,\beta\in \I$, $|\alpha|<|\beta|$, and $\gamma\in \alpha\vee \beta$, then there is an atom $b\leq \beta$ such that $b\not\leq \alpha$ and $(\alpha\vee b)\cap S_{\leq \gamma}\subseteq \I$.
\item[\mylabel{I4}{\textbf{I4}}] If $\alpha,\beta\in S$, $\nu\in\max(\I\cap S_{\leq \alpha})$, and $\nu\leq \beta$, then $\alpha\vee \beta\neq\varnothing$.
\end{enumerate}
\end{defn}

Roughly speaking, conditions (I1) and (I2) say that $\I$ is an \textit{order ideal} of $S$; condition (I3) guarantees a \textit{local} matroid structure on each Boolean lattice $S_{\leq \gamma}$; and condition (I4) ensures reasonable gluing of these local matroids. There is a stronger (global) version of (I3) given in \cite[Thm. 6.2]{bibby} that one can adopt to define a \textit{matroid scheme}, related to the global geometric condition in the last section, but this is not used in the present work. 

Whenever the simplicial poset is a lattice (resp. semilattice), a matroid prescheme $(S,\I)$ is equivalent to a matroid (resp. semimatroid as in \cite{ardila}). 
In any matroid prescheme, the restriction of $\I$ to a Boolean lattice $S_{\leq \gamma}$ defines the independent sets of a matroid on the ground set $\at(\gamma)$.

\begin{example}\label{ex:M2}
Consider the simplicial poset $S$ whose Hasse diagram is depicted in \Cref{fig:ex2}. The pair $(S,\I)$, where $\I$ is the collection of elements outlined in rectangles, forms a matroid prescheme. 
\end{example}

\begin{figure}[ht]
\centering
\begin{tikzpicture}[scale=1.7]
\node[draw,rectangle,very thick] (0) at (0,.3) {\scriptsize$\hat{0}$};
\foreach \x in {1,2,3,4}
{\node[draw,rectangle,very thick] (\x) at (-2.5+\x,1) {\scriptsize\x};
\draw[-] (0.north)--(\x.south);};
\node[draw,rectangle,very thick] (12) at (-2.5,1.8) {\scriptsize$(c,12)$};
\node[draw,rectangle,very thick] (13) at (-1.5,1.8) {\scriptsize$(c,13)$};
\node[draw,rectangle] (23) at (-.5,1.8) {\scriptsize$(c,23)$};
\node[draw,rectangle,very thick] (14) at (.5,1.8) {\scriptsize$(x,14)$};
\node[draw,rectangle,very thick] (24) at (1.5,1.8) {\scriptsize$(y,24)$};
\node[draw,rectangle,very thick] (34) at (2.5,1.8) {\scriptsize$(z,34)$};
\node[draw,rectangle,very thick] (1124) at (-3,3) {\scriptsize$(u_1,124)$};
\node[draw,rectangle,very thick] (1134) at (-2,3) {\scriptsize$(u_1,134)$};
\node[draw,rectangle] (1234) at (-1,3) {\scriptsize$(u_1,234)$};
\node (123) at (0,3) {\scriptsize$(c,123)$};
\node[draw,rectangle,very thick] (2124) at (1,3) {\scriptsize$(u_2,124)$};
\node[draw,rectangle,very thick] (2134) at (2,3) {\scriptsize$(u_2,134)$};
\node[draw,rectangle] (2234) at (3,3) {\scriptsize$(u_2,234)$};
\node (u1) at (-1.5,3.7) {\scriptsize$(u_1,1234)$};
\node (u2) at (1.5,3.7) {\scriptsize$(u_2,1234)$};
\foreach \x in {1124,1134,1234,123} {\draw[-] (u1.south)--(\x.north);};
\foreach \x in {2124,2134,2234,123} {\draw[-] (u2.south)--(\x.north);};
\foreach \x in {12,13,14} {\draw[-] (1.north)--(\x.south);};
\foreach \x in {12,23,24} {\draw[-] (2.north)--(\x.south);};
\foreach \x in {13,23,34} {\draw[-] (3.north)--(\x.south);};
\foreach \x in {14,24,34} {\draw[-] (4.north)--(\x.south);};
\foreach \x in {1124,123,2124} {\draw[-] (12.north)--(\x.south);};
\foreach \x in {1134,123,2134} {\draw[-] (13.north)--(\x.south);};
\foreach \x in {1234,123,2234} {\draw[-] (23.north)--(\x.south);};
\foreach \x in {1124,2124,1134,2134} {\draw[-] (14.north)--(\x.south);};
\foreach \x in {1124,2124,1234,2234} {\draw[-] (24.north)--(\x.south);};
\foreach \x in {1134,2134,1234,2234} {\draw[-] (34.north)--(\x.south);};
\end{tikzpicture}
\caption{An example of a matroid scheme, with the independent elements in the simplicial poset outlined in rectangles (see \Cref{ex:M2}). 
The bolder rectangles mark those independent elements which are $\nbc$ (see later \Cref{ex:nbc}).}
\label{fig:ex2}
\end{figure}
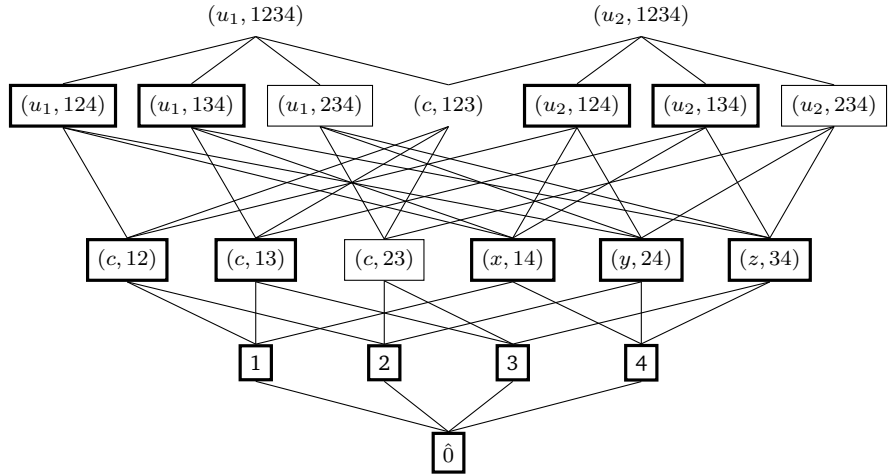

\begin{defn}\label{def:rank}
The \textbf{rank function} $\rho=\rho_M\colon S\to\Z_{\geq0}$ of a matroid prescheme $M=(S,\I)$ is 
\[\rho(\alpha):=\max\{|\nu|: \nu\in\I\cap S_{\leq \alpha}\}.\]
\end{defn}

As stated in \cite[Prop 4.8]{bibby}, the rank function on a matroid prescheme $M=(S,\I)$ satisfies the following useful condition:
\begin{enumerate}
\item[\mylabel{M4'}{\textbf{M4'}}] If $\alpha,\beta\in S$, $\nu\in\alpha\meet\beta$, and $\rk(\alpha)=\rk(\nu)$, then $\alpha\join\beta\neq\varnothing$ and $\rk(\gamma)=\rk(\beta)$ for any $\gamma\in\alpha\join\beta$.
\end{enumerate}

The next definitions are equivalent to those in \cite[Defns 5.5, 8.1, 11.4]{bibby}.  

\begin{defn}\label{def:c-f}
Let $M=(S,\I)$ be a matroid prescheme.
A \textbf{circuit} is a minimal element of $S\setminus \I$.
An element $x\in S$ is called a \textbf{flat} if, for all circuits $\zeta$: if there exists an element $\alpha\leq x$ such that $\zeta$ covers $\alpha$, then $\zeta\leq x$. The \textbf{poset of flats} is the subposet of $S$ whose elements are the flats of $M$.
Say that $M$ is \textbf{simple} if every atom is a flat and not a circuit.
\end{defn}

Let us quickly record the useful \emph{circuit elimination axiom} from \cite[Thm 8.3]{bibby}.
\begin{enumerate}
\item[\mylabel{C3}{\textbf{C3}}] 
if $\zeta_1,\zeta_2$ are distinct circuits, $\omega\in \zeta_1\join\zeta_2$, and $a\in\at(\zeta_1)\cap\at(\zeta_2)$, then there exists a circuit $\xi$ with $\xi\leq\omega\setminus a$.
\end{enumerate}

\begin{example}\label{ex:CF}
    In \Cref{fig:ex2}, $(c,123)$ is the only circuit.
    The flats are given by $\hat{0}$, the atoms $1$, $2$, $3$, and $4$, the rank-$2$ elements $(x,14)$, $(y,24)$, and $(z,34)$, as well as $(c,123)$, $(u_1,1234)$, and $(u_2,1234)$.
    Notice that the poset of flats of this matroid prescheme is isomorphic to the geometric poset  
    in \Cref{fig:P2}, via $(x,\at(x))\mapsto x$.
    In fact, every locally geometric poset arises in this way, as indicated by the following. 
    \end{example}

\begin{theorem}[{\cite[Thm 11.6]{bibby}}]
    A poset is locally geometric if and only if it is isomorphic to the poset of flats of a matroid prescheme. Furthermore, each locally geometric poset is the poset of flats of a unique simple matroid prescheme, up to isomorphism.
\end{theorem}

The proof explicitly constructs, for any locally geometric poset $\P$, a simple matroid prescheme $M(\P)$ whose poset of flats is isomorphic to $\P$. This construction is as follows.

\begin{defn}
The \textbf{associated matroid prescheme} $M(\P)=(S,\I)$ of a locally geometric poset $\P$ consists of the simplicial poset
    \[S=S(\P)=\{(x,T) \colon T\subseteq\at(\P), x\in\vee T\}\]
    viewed as a subposet of $\P\times 2^{\at(\P)}$,
    and the independence poset
    \[\I=\I(\P)=\{(x,T)\in S \colon \rk_\P(x)=|T|\}.\]
    The rank function on $M(\P)$ is given by $\rk_M(x,T)=\rk_\P(x)$.
\end{defn}

\begin{example}\label{ex:M(P)}
The matroid prescheme $M(\P)$ associated to the locally geometric poset $\P$ in \Cref{fig:P2} is precisely the matroid prescheme in \Cref{fig:ex2}, where we abbreviate the atoms $(i,i)$ by $i$.
\end{example}

\subsection{Unicircular elements}\label{sec:uni}
In this section, we introduce the notion of a unicircular element and establish useful properties. These unicircular elements will play a crucial role in \Cref{sec:OSdef,sec:gb} while working with the Orlik--Solomon ideal.

\begin{defn}
In a matroid prescheme $(S,\I)$, an element $\omega\in S$ for which there exists a unique circuit $\zeta$ with $\zeta\leq \omega$ is called \textbf{unicircular}; we may write $\zeta$-\textbf{unicircular} to distinguish this unique circuit.
\end{defn}

\begin{example}
 In \Cref{fig:ex2}, $\zeta=(c,123)$ is the only circuit (see \Cref{ex:CF}).
    As a circuit, $\zeta$ is necessarily
    $\zeta$-unicircular. 
    Additionally, $(u_1,1234)$ and $(u_2,1234)$ are $\zeta$-unicircular.
    \end{example}

We now characterize unicircular elements and establish properties that will be needed later.

\begin{prop}\label{lem:u-a-ind}
Let $(S,\I)$ be a matroid prescheme with $\omega\in S\setminus\I$ and $a\in\at(\omega)$.
Then $\omega\setminus a$ is independent if and only if there exists a circuit $\zeta$ for which $\omega$ is $\zeta$-unicircular and $a\leq\zeta$.
\end{prop}
\begin{proof}

First, suppose $\omega\setminus a$ is independent. 
Since $\omega$ is not independent, there exists a circuit $\zeta\leq\omega$. 
By independence of $\omega\setminus a$, any circuit $\zeta'\leq\omega$ must have $a\in\at(\zeta')$. 
Then by \eqref{C3}, there must be a unique such circuit.
Therefore, $\omega$ is $\zeta$-unicircular and $a\leq\zeta$.

Conversely, suppose $\omega$ is $\zeta$-unicircular with $a\leq \zeta$.
If $\omega\setminus a$ were not independent, there would be a circuit $\zeta'\leq\omega\setminus a\leq\omega$.
Uniqueness of $\zeta$ would imply $\zeta=\zeta'$, but this is impossible since $a\leq \zeta$ and $a\not\leq\zeta'$.
Thus, $\omega\setminus a$ must be independent.
\end{proof}

\begin{prop}[Rank characterizes unicircular]\label{prop:rank-uni}
In a matroid prescheme $(S,\I)$ with rank function $\rk$, an element $\omega\in S$ is unicircular if and only if $\rk(\omega)=|\omega|-1$.
\end{prop}
\begin{proof}

First suppose $\omega$ is $\zeta$-unicircular, and let $a\in\at(\zeta)$. Then $\omega\setminus a$ is independent by 
\Cref{lem:u-a-ind},
so 
$|\omega|-1=|\omega\setminus a|=\rk(\omega\setminus a)\leq\rk(\omega)<|\omega|$ and hence $\rk(\omega)=|\omega|-1$.

Conversely, suppose $\rk(\omega)=|\omega|-1$, and let $\nu\in\I_{\leq \omega}$ be such that $\rk(\nu)=\rk(\omega)$. 
Then $|\nu|=|\omega|-1$, meaning $a=\omega\setminus\nu$ is an atom.
Now as $\nu=\omega\setminus a$ is independent,
$\omega$ is unicircular by \Cref{lem:u-a-ind}.
\end{proof}

\begin{lemma}\label{lem:u-join}
Let $(S,\I)$ be a matroid prescheme with $\zeta$-unicircular  $\omega\in S$,  $a\in\at(\zeta)$, and $\nu\in S$ such that $\omega\setminus a\in\omega\meet\nu$. 
Then $\omega\vee\nu\neq\varnothing$ and, if $\nu\in\I$, then any element of $\omega\vee\nu$ is $\zeta$-unicircular.
\end{lemma}
\begin{proof}
As $\omega\setminus a$ is independent by \Cref{lem:u-a-ind}, $\rk(\omega\setminus a)=|\omega\setminus a|=|\omega|-1=\rk(\omega)$. Then $\omega\join\nu\neq\varnothing$ and any $\gamma\in\omega\vee\nu$ has $\rk(\gamma)=\rk(\nu)$ by \eqref{M4'}. If, in addition, $\nu\in\I$, then  $\rk(\gamma)=\rk(\nu)=|\nu|=|\gamma|-1$ and $\gamma$ is $\zeta$-unicircular by \Cref{prop:rank-uni}. 
\end{proof}

\begin{lemma}\label{lem:i'-iii'}
Let $(S,\I)$ be a matroid prescheme with independent $\beta\in\I$ and $\zeta$-unicircular $\omega\in S$. 
If $a\in\at(\zeta)\cap\at(\beta)$ such that $(\omega\setminus a)\meet\beta=\{\zero\}$, then  $((\omega\setminus a)\join\beta)\cap\I=\varnothing$.
\end{lemma}
\begin{proof}
Suppose that $\gamma\in(\omega\setminus a)\join\beta$. 
Either $\gamma\geq\omega$ or $\omega\setminus a\in\gamma\meet\omega$,
so there exists some $\upsilon\in\omega\join\gamma$ (using \Cref{lem:u-join} in the latter case).

Now, $\zeta$ is the unique join of its atoms in the lattice $S_{\leq\upsilon}$, and similarly for $\gamma$.
 Since $\at(\zeta)\subseteq\at(\gamma)$, this implies $\zeta\leq\gamma$
 and thus $\gamma$ cannot be independent via \eqref{I2}.
\end{proof}

\begin{lemma}\label{lem:ii'-iii'}
Let $(S,\I)$ be a matroid prescheme with independent $\beta\in\I$ and $\zeta$-unicircular  $\omega\in S$. If $\omega\meet\beta=\{\zero\}$, then for any $a\in\at(\zeta)$,
\[ 
((\omega\setminus a)\join\beta)\cap\I=
\{\upsilon\setminus a\st\upsilon\in\omega\vee\beta \text{ and $\upsilon$ is $\zeta$-unicircular}\}.
\]
\end{lemma}
\begin{proof}
Suppose $\gamma\in((\omega\setminus a)\vee\beta)\cap\I$. 
Then $\gamma\meet\omega\ni\omega\setminus a$, 
so there exists some $\zeta$-unicircular element $\upsilon\in\gamma\join\omega$ by \Cref{lem:u-join}. 
As $S_{\leq\upsilon}$ is a Boolean lattice and $\at(\gamma)=\at(\upsilon)\setminus \{a\}$, we conclude $\gamma=\upsilon\setminus a$.

Conversely, suppose $\upsilon\in\omega\join\beta$ and $\upsilon$ is $\zeta$-unicircular. Then $\upsilon\setminus a$ is independent by \Cref{lem:u-a-ind} and, 
 because $S_{\leq\upsilon}$ is a Boolean lattice, $((\omega\setminus a)\join\beta)\cap S_{\leq \upsilon} = \{\upsilon\setminus a\}$. 
\end{proof}

\begin{lemma}\label{lem:iii'-iii'}
Let $(S,\I)$ be a matroid prescheme with $\omega_i$ a $\zeta_i$-unicircular element for $i=1,2$, with $\omega_1\neq\omega_2$, and assume that $\omega_1\setminus a_1=\omega_2\setminus a_2$ for some atoms $a_i\in\at(\zeta_i)$. 
Then $\omega_1\vee\omega_2\neq\varnothing$ and, for all $\upsilon\in\omega_1\vee\omega_2$ and $a\in\at(\zeta_1)\cup\at(\zeta_2)$, there exists a circuit $\zeta^\upsilon_a$ so that
\begin{enumerate}
\item  
$\upsilon\setminus a$ is $\zeta^\upsilon_a$-unicircular,
\item $\at(\zeta^\upsilon_a)\subseteq \at(\zeta_1)\cup\at(\zeta_2)$, and
\item if $b\in\at(\zeta_1)\cup\at(\zeta_2)$ with $b\in\at(\zeta^\upsilon_a)$, then $a\in\at(\zeta^\upsilon_b)$.
\end{enumerate}
\end{lemma}
\begin{proof}
Since $\omega_1\setminus a_1\in\omega_1\meet\omega_2$,  $\omega_1\join\omega_2\neq\varnothing$ by \Cref{lem:u-join}. Let $\upsilon\in\omega_1\join\omega_2$ and  $a\in\at(\zeta_1)\cup\at(\zeta_2)$.

Without loss of generality, assume $a\in\at(\zeta_1)$. 
Then $(\upsilon\setminus a)\setminus a_2=\omega_1\setminus a$ is independent, hence $\upsilon\setminus a$ is $\zeta^\upsilon_a$-unicircular for some circuit $\zeta^\upsilon_a$, by \Cref{lem:u-a-ind}.
If $a\notin\at(\zeta_2)$, then $\zeta^\upsilon_a=\zeta_2$ and (2) holds.
On the other hand, if $a\in\at(\zeta_1)\cap\at(\zeta_2)$, then let 
$\gamma\in\zeta_1\join\zeta_2$ such that $\gamma\leq\upsilon$. By \eqref{C3}, there exists a circuit $\zeta^\upsilon_a\leq \gamma\setminus a\leq\upsilon\setminus a$,
which necessarily satisfies (2). 

Now suppose $b\in\at(\zeta_1)\cup\at(\zeta_2)$ with $b\in\at(\zeta^\upsilon_a)$.
If $a\notin\at(\zeta^\upsilon_b)$, then $\zeta^\upsilon_b\leq\upsilon\setminus a$, contradicting $\zeta^\upsilon_a$-unicircularity of $\upsilon\setminus a$.
Thus, $a\in\at(\zeta^\upsilon_b)$.
\end{proof}

\subsection{No-broken-circuit elements}\label{sec:nbc}

In this section, we define and characterize \textit{no-broken-circuit ($\nbc$) elements} in a matroid prescheme. The $\nbc$ elements play an important role in \S\ref{sec:nbcbasis}.

\begin{defn}\label{defn:nbc}
Let $M=(S,\I)$ be a matroid prescheme with a linear order $\ato$ on the atoms of $S$.
An element of $S$ is called a \textbf{broken circuit} if it is of the form $\zeta\setminus a$ for some circuit $\zeta$ with $a=\min_{\ato}\at(\zeta)$.
An element $\beta\in S$ is called a \textbf{no-broken-circuit} element, or $\nbc$ for short, if there are no broken circuits in $S_{\leq\beta}$. We denote $\nbc(M)$ as the subposet of $S$ containing precisely the $\nbc$ elements.

When $M=M(\P)$ for a locally geometric poset $\P$, we write $\nbc(\P):=\nbc(M)$.
\end{defn}

\begin{example}\label{ex:nbc}
In \Cref{fig:ex2} with $1\ato2\ato3\ato4$, the only broken circuit is $(c,23)=(c,123)\setminus 1$. As such, the $\nbc$ elements are exactly the elements not above $(c,23)$, which are all boxed in bold.
\end{example}

The next few lemmas record basic but useful properties of $\nbc$ elements.
\begin{lemma}[$\nbc$ is independent]
\label{lem:nbc-ind}
If $(S,\I)$ is a matroid prescheme  with linear order $\ato$ on the atoms of $S$,  then $\nbc(M)\subseteq\I$
\end{lemma}
\begin{proof}
If $\beta$ is not independent, then $\beta\geq\zeta\geq\zeta\setminus a$ for some circuit $\zeta$ and $a=\min_{\ato}\at(\zeta)$.
\end{proof}

\begin{lemma}[$\nbc$ is local]\label{lem:localnbc}
Let $\P$ be a locally geometric poset with linear order $\ato$ on $\at(\P)$, and let $x\in\P$ with induced order on $\at(x)$. Then
$\beta\in\nbc(\P_{\leq x})$ if and only if $\beta\in\nbc(\P)$ and $\beta\leq(x,\at(x))$.
\end{lemma}
\begin{proof}
Write $M(\P)=(S,\I)$ and $\alpha=(x,\at(x))$, so that $M(\P_{\leq x})=(S_{\leq\alpha},\I_{\leq \alpha})$. It suffices to prove that an element $\beta\in S_{\leq \alpha}$ is a broken circuit in $M(\P_{\leq x})$ if and only if it is a broken circuit in $M(\P)$. 
The forward direction holds because every circuit of $M(\P_{\leq x})$ is also a circuit of $M(\P)$.
Conversely, suppose that $\beta=\zeta\setminus a$ where $\zeta$ is a circuit in $M(\P)$ and $a=\min_{\ato}\at(\zeta)$. 
Then $\zeta\leq\alpha$ by \Cref{def:c-f}, since $\alpha$ is a flat of $M(\P)$, and therefore $\beta=\zeta\setminus a$ is a broken circuit of $M(\P_{\leq x})$.
\end{proof}

\begin{lemma}\label{lem:nbc-uni}
Let $(S,\I)$ be a matroid prescheme with linear order $\ato$ on the atoms of $S$, and let $\beta\in\I$. Then $\beta$ is $\nbc$ if and only if there is no $\zeta$-unicircular element $\omega$ such that  $a=\min_{\ato}\at(\zeta)$ and $\beta=\omega\setminus a$.
\end{lemma}
\begin{proof}
If $\omega$ is $\zeta$-unicircular and $a=\min_{\ato}\at(\zeta)$, then $\omega\setminus a\geq\zeta\setminus a$ hence $\omega\setminus a$ is  not $\nbc$.
Conversely, assume that $\beta\in\I$ is not $\nbc$. Then there exists a circuit $\zeta$ with $a=\min_{\ato}\at(\zeta)$ such that $\beta\geq\zeta\setminus a$.
Now $a\not\leq\beta$, since otherwise $\beta\geq\zeta$ contradicts independence, 
so $\zeta\setminus a\in\zeta\meet\beta$.
Then by \Cref{lem:u-join}, there exists some $\omega\in\zeta\join\beta$ and $\omega$ is $\zeta$-unicircular. 
Finally, $\beta=\omega\setminus a$ because $\at(\omega)=\at(\beta)\cup\{a\}$.
\end{proof}

\subsection{Deletion and contraction}\label{sec:del-con}
In this section, we consider the matroid operations of deletion and contraction and demonstrate the relationship between their $\nbc$ elements.

\begin{prop}
Let $M=(S,\I)$ be a simple matroid prescheme, and let $a$ be a distinguished atom of $S$.
Then $M'=(S_{\not\geq a},\I_{\not\geq a})$ is a matroid prescheme called the \textbf{deletion}, and $M''=(S_{\geq a}, \I_{\geq a})$ is a matroid prescheme called the \textbf{contraction}. 
Call $(M,M',M'')$ a \textbf{triple of matroid preschemes}.
\end{prop}
\begin{proof}
Note that $\I=\{\nu\in S\st \rk(\nu)=|\nu|\}$ where $\rk\colon S\to\Z_{\geq0}$ is the rank function of $M$ from \Cref{def:rank}. 
The deletion is defined in \cite[Defn 10.2]{bibby} by its rank function $\rk'=\rk|_{S\not\geq a}\colon S_{\not\geq a}\to\Z_{\geq 0}$. 
By \cite[Thm 6.2]{bibby}, it suffices to check that $\I_{\not\geq a} = \{\nu\in S_{\not\geq a} \st \rk'(\nu)=|\nu|\}$.
But this is true because, for $\nu\in S_{\not\geq a}$,  $\rk'(\nu)=|\nu|$ if and only if $\rk(\nu)=|\nu|$.
The contraction is defined in \cite[Defn 10.5]{bibby} by its rank function $\rk''\colon S_{\geq a}\to\Z_{\geq 0}$ with $\rk''(\beta)=\rk(\beta)-\rk(a)=\rk(\beta)-1$ (since $M$ is simple, $a\in\I$). Here, it suffices to check  $\I_{\geq a} = \{\nu\in S_{\geq a} \st \rk''(\nu)=|\nu|_{S_{\geq a}} = |\nu|_S-1\}$. 
But this is true because, for $\nu\in S_{\geq a}$, $\rk''(\nu)=|\nu|_S-1$ if and only if $\rk(\nu)=|\nu|_S$.
\end{proof}

\begin{example}
Using the matroid prescheme $M$ from \Cref{fig:ex2} and distinguished atom 4, the matroid preschemes $M'$ and $M''$ are shown in \Cref{fig:M',fig:M''}, respectively. In each, the independent elements are boxed. 
Note that the linear order $1\ato2\ato3\ato4$ on the atoms of $S$ induces the order $1\ato2\ato3$ on the atoms of $S_{\not\geq 4}$, and it induces the order $(x,14)\ato(y,24)\ato(z,24)$ on the atoms of $S_{\geq 4}$.
The corresponding $\nbc$ elements are boxed in bold in \Cref{fig:M',fig:M''}.
Observe that, as the next lemma states, $\nbc(M)=\nbc(M')\sqcup\nbc(M'')$.

\begin{figure}[ht]
\begin{subfigure}[t]{.35\textwidth}
\centering
\begin{tikzpicture}
\node[draw,rectangle,very thick] (0) at (0,0) {\scriptsize$\hat{0}$};
\node[draw,rectangle,very thick] (1) at (-1.5,1) {\scriptsize1};
\node[draw,rectangle,very thick] (2) at (0,1) {\scriptsize2};
\node[draw,rectangle,very thick] (3) at (1.5,1) {\scriptsize3};
\foreach \x in {1,2,3} {\draw[-] (0.north)--(\x.south);};
\node[draw,rectangle,very thick] (12) at (-1.5,2.5) {\scriptsize$(c,12)$};
\node[draw,rectangle,very thick] (13) at (0,2.5) {\scriptsize$(c,13)$};
\node[draw,rectangle] (23) at (1.5,2.5) {\scriptsize$(c,23)$};
\node (123) at (0,3.5) {\scriptsize$(c,123)$};
\foreach \x in {12,13} {\draw[-] (1.north)--(\x.south);};
\foreach \x in {12,23} {\draw[-] (2.north)--(\x.south);};
\foreach \x in {13,23} {\draw[-] (3.north)--(\x.south);};
\foreach \x in {12,13,23} {\draw[-] (\x.north)--(123.south);};
\end{tikzpicture}
\caption{$M'$}
\label{fig:M'}
\end{subfigure}
\begin{subfigure}[t]{.6\textwidth}
\centering
\begin{tikzpicture}
\node[draw,rectangle,very thick] (4) at (0,0) {\scriptsize4};
\node[draw,rectangle,very thick] (14) at (-1.5,1) {\scriptsize$(x,14)$};
\node[draw,rectangle,very thick] (24) at (0,1) {\scriptsize$(y,24)$};
\node[draw,rectangle,very thick] (34) at (1.5,1) {\scriptsize$(z,34)$};
\node[draw,rectangle,very thick] (1124) at (-4,2.5) {\scriptsize$(u_1,124)$};
\node[draw,rectangle,very thick] (1134) at (-2.5,2.5) {\scriptsize$(u_1,134)$};
\node[draw,rectangle] (1234) at (-1,2.5) {\scriptsize$(u_1,234)$};
\node[draw,rectangle,very thick] (2124) at (1,2.5) {\scriptsize$(u_2,124)$};
\node[draw,rectangle,very thick] (2134) at (2.5,2.5) {\scriptsize$(u_2,134)$};
\node[draw,rectangle] (2234) at (4,2.5) {\scriptsize$(u_2,234)$};
\node (u1) at (-2.5,3.5) {\scriptsize$(u_1,1234)$};
\node (u2) at (2.5,3.5) {\scriptsize$(u_2,1234)$};
\foreach \x in {1124,1134,1234} {\draw[-] (u1.south)--(\x.north);};
\foreach \x in {2124,2134,2234} {\draw[-] (u2.south)--(\x.north);};
\foreach \x in {14,24,34} {\draw[-] (4.north)--(\x.south);};
\foreach \x in {1124,2124,1134,2134} {\draw[-] (14.north)--(\x.south);};
\foreach \x in {1124,2124,1234,2234} {\draw[-] (24.north)--(\x.south);};
\foreach \x in {1134,2134,1234,2234} {\draw[-] (34.north)--(\x.south);};
\end{tikzpicture}
\caption{$M''$}
\label{fig:M''}
\end{subfigure}
\caption{The deletion and contraction of $M$ from \Cref{fig:ex2} with respect to atom 4.}
\end{figure}
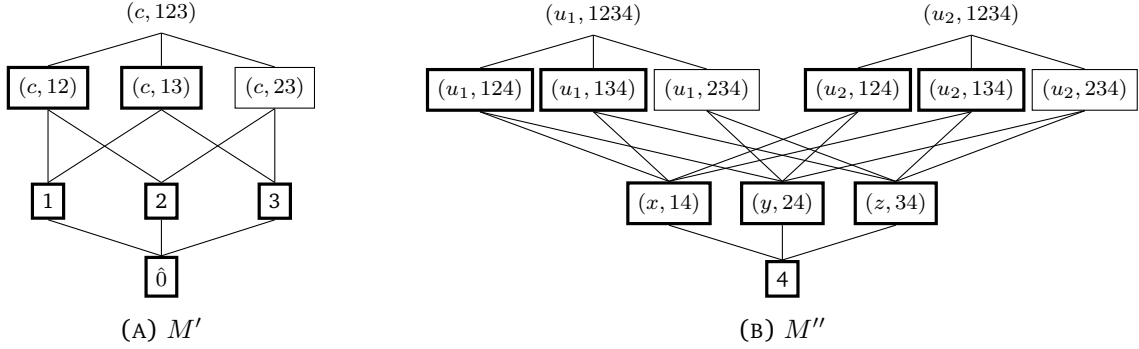
\end{example}

\begin{prop}\label{lem:nbc}
Let $M=(S,\I)$ be a simple matroid prescheme with linear order $\ato$ on the atoms of $S$.
Let $(M,M',M'')$ be the triple of matroid preschemes with respect to the distinguished atom $a=\max_{\ato}\at(S)$,
with linear order on the atoms of $S_{\geq a}$ and $S_{\not\geq a}$ induced by $\ato$.
    Then the set of $\nbc$-elements of $M$ is the disjoint union of $\nbc$-elements of the deletion $M'$ and $\nbc$-elements of the contraction $M''$:
    \[ \nbc(M) = \nbc(M')\sqcup \nbc(M'')\]
\end{prop}
\begin{proof}
Since $\I=\I_{\geq a}\sqcup \I_{\not\geq a}$ and $\nbc$ elements are independent (\Cref{lem:nbc-ind}), it suffices to show that $\nbc(M')=\nbc(M)\cap\I_{\not\geq a}$ and $\nbc(M'')=\nbc(M)\cap\I_{\geq a}$.

First consider $\beta\in\I_{\not\geq a}$.
Suppose $\beta\notin\nbc(M')$. Then by \Cref{lem:nbc-uni}, there exists a $\zeta$-unicircular element $\omega$ in $M'$ such that $\beta=\omega\setminus a_0$ with $a_0=\min_{\ato}\at(\zeta)$.
Then, using \Cref{prop:rank-uni}, $\rk(\omega)=\rk'(\omega)=|\omega|_{S\not\geq a}-1=|\omega|_{S}-1$ and thus $\omega$ is $\zeta$-unicircular in $M$. Consequently, $\beta=\omega\setminus a_0\notin\nbc(M)$.
Conversely, suppose $\beta\notin\nbc(M)$.
Then by \Cref{lem:nbc-uni}, there exists a $\zeta$-unicircular element $\omega$ in $M$ such that $\beta=\omega\setminus a_0$ with $a_0=\min_{\ato}\at(\zeta)$. 
By maximality of $a$, $a\neq a_0$.
This implies $\omega\in S_{\not\geq a}$ and, since $\rk'(\omega)=\rk(\omega)=|\omega|-1$, $\omega$ is $\zeta$-unicircular in $M'$. 
Thus, $\beta=\omega\setminus a_0\notin\nbc(M')$.

Now consider $\beta\in\I_{\geq a}$. 
Suppose $\beta\notin\nbc(M'')$. 
Then by \Cref{lem:nbc-uni}, there exists a $\zeta$-unicircular element $\omega$ in $M''$ such that $\beta=\omega\setminus a_0$ with $a_0=\min_{\ato}\at(\zeta)$.
Then, using \Cref{prop:rank-uni}, $\rk(\omega)=\rk''(\omega)+1=|\omega|_{S_{\geq a}} = |\omega|_S-1$ and thus $\omega$ is $\zeta$-unicircular in $M$.
Consequently, $\beta=\omega\setminus a_0\notin\nbc(M)$.
Conversely, suppose $\beta\notin\nbc(M)$. 
Then by \Cref{lem:nbc-uni}, there exists a $\zeta$-unicircular element $\omega$ in $M$ such that $\beta=\omega\setminus a_0$ with $a_0=\min_{\ato}\at(\zeta)$. Now, $\omega\in S_{\geq a}$ because $\omega\geq\beta\geq a$ and, since $\rk''(\omega)=\rk(\omega)-1=|\omega|_S-2=|\omega|_{S\geq a}-1$, $\omega$ is $\zeta$-unicircular in $M''$. Thus, $\beta\notin\nbc(M'')$.
\end{proof}

\section{The Orlik--Solomon algebra}\label{sec:OS}
\subsection{The construction of the Orlik--Solomon algebra}\label{sec:OSdef}
Let $\P$ be a locally geometric poset, $(S,\I)$ its associated matroid prescheme, and $E=E(\P)$
the graded-commutative $\mathbb{Z}$-algebra generated by $e_\alpha$ for $\alpha\in\mathcal{I}\setminus\{\hat0\}$, where $\deg(e_\alpha)=|\alpha|$.
By convention, write $e_{\zero}:=1\in E$. Fix a linear order on the atoms of $\I$, and write $\at(\alpha)$ in increasing order accordingly.
We may then define the function $\sgn:\mathcal{I}\times\mathcal{I}\to\{-1,1\}$ which takes $(\alpha,\beta)\in\I\times\I$ to the sign of the permutation needed to sort $(\at(\alpha),\at(\beta))$ into ascending order. Since all $\gamma\in\alpha\join\beta$ share the same set of atoms, given $\alpha,\beta,\nu\in\I$ we may write $\sgn(\alpha\join\beta,\nu)$ and this is well-defined.

\begin{defn}[OS Algebra]\label{defn:OS}
    Let $J\subseteq E$ be the ideal generated by the following:
    \begin{enumerate}
    \item[\mylabel{OSi}{i}] $e_\alpha e_\beta$ when $\alpha\meet\beta\neq\{\hat{0}\}$;
    \item[\mylabel{OSii}{ii}] $e_\alpha e_\beta-\sgn(\alpha,\beta)\sum\limits_{\gamma\in(\alpha\join\beta)\cap\I}e_\gamma$ when $\alpha\meet\beta=\{\hat{0}\}$; and
    \item[\mylabel{OSiii}{iii}] $\sum\limits_{a\in\at(\zeta)}\sgn(a,\omega\setminus a)e_{\omega\setminus a}$ when $\omega$ is a $\zeta$-unicircular element.
    \end{enumerate}
    We call $J$ the Orlik--Solomon ideal of $\mathcal{P}$, and $A(\mathcal{P}):=E/J$ the \textbf{Orlik--Solomon algebra} of $\mathcal{P}$.
\end{defn}

\begin{example}\label{ex:OS}
 Consider the matroid prescheme shown in \Cref{fig:ex2}. In this case, $E$ has 16 generators:
 $e_1$, $e_2$, $e_3$, and $e_4$ in degree 1; 
 $e_{(c,12)}$, $e_{(c,13)}$, $e_{(c,23)}$, $e_{(x,14)}$, $e_{(y,24)}$, and $e_{(z,34)}$ in degree 2; 
 and $e_{(u_1,124)}$, $e_{(u_1,134)}$, $e_{(u_1,234)}$, $e_{(u_2,124)}$, $e_{(u_2,134)}$, $e_{(u_2,234)}$ in degree 3. 
 The ideal $J$ is generated by 
 over 100
 elements; we will point out a few illustrative ones.
 For example, the type \eqref{OSii} generator  
\begin{equation}\label{eq:ex2:ii}
    e_4e_{(c,23)}-e_{(u_1,234)}-e_{(u_2,234)}
\end{equation}
arises from $4\vee(c,23)=\{(u_1,234),(u_2,234)\}$, so that in the quotient we have the product
$e_4e_{(c,23)} = e_{(u_1,234)}+e_{(u_2,234)}$. 
The two unicircular elements $(u_1,1234)$ and $(u_2,1234)$, which share the circuit $(c,123)$ and set of atoms $\{1,2,3,4\}$, give rise to the type \eqref{OSiii} generators 
    \begin{equation}\label{eq:ex2:iii}
    e_{(u_1,234)}-e_{(u_1,134)}+e_{(u_1,124)} \qquad \text{and} \qquad
   e_{(u_2,234)}-e_{(u_2,134)}+e_{(u_2,124)}. 
   \end{equation}
Additionally, the circuit $(c,123)$ is unicircular itself, and thus gives rise to its own type \eqref{OSiii} generator
\begin{equation}\label{eq:ex2:iiic}
    e_{(c,23)}-e_{(c,13)}+e_{(c,12)}.
\end{equation}
\end{example}

The following provides an alternative presentation that considers all dependent elements. This presentation, perhaps, makes the relationship to the classical case (\Cref{prop:classicalOS}) more apparent, but it is less efficient, meaning that it has many more generators and relations. 

\begin{prop}[Alternative presentation]\label{prop:altOS}
Let $\P$ be a locally geometric poset. Then the graded algebra $A(\P)$ is isomorphic to the quotient $\bar{E}/\bar{J}$ where $\bar{E}$ is the graded-commutative algebra generated by $\bar{e}_\alpha$ in degree $|\alpha|$ for $\alpha\in S\setminus\{\zero\}$, and $\bar{J}\subseteq\bar{E}$ is the ideal generated by
\begin{enumerate}
\item[(0)] $\bar{e}_\alpha$ whenever $\alpha\notin\I$,
\item[($\bar{\imath}$)] $\bar{e}_\alpha\bar{e}_\beta$ whenever $\alpha\meet\beta\neq\{\zero\}$,
\item[($\ol{\imath\imath}$)] $\bar{e}_\alpha\bar{e}_\beta-\sgn(\alpha,\beta)\displaystyle\sum_{\gamma\in\alpha\vee\beta} \bar{e}_\gamma$ whenever $\alpha\meet\beta=\{\zero\}$, and
\item[($\ol{\imath\imath\imath}$)]
$\displaystyle\sum_{a\in\at(\omega)} \sgn(a,\omega\setminus a) \bar{e}_{\omega\setminus a}$ whenever $\omega\notin\I$.
\end{enumerate}
The isomorphism is induced by the inclusion $E\into\bar{E}$ with $e_\alpha\mapsto\bar{e}_\alpha$.
\end{prop}
\begin{proof}
Let $j\colon E\into\bar{E}$ be the inclusion, and let $\bar{J}_0:=\langle\bar{e}_\alpha\st\alpha\notin\I\rangle\subseteq\bar{J}\subseteq\bar{E}$.
Using the isomorphism $E\cong \bar{E}/\bar{J}_0$, it suffices to see that $\bar{J}=j(J)+\bar{J}_0$.
Consider $\alpha,\beta\in S$.
If either $\alpha\notin\I$ or $\beta\notin\I$, then $\bar{e}_\alpha\bar{e}_\beta\in\bar{J}_0$ and $\bar{e}_\gamma\in\bar{J}_0$ for any $\gamma\in\alpha\join\beta$.
If $\alpha\in\I$ and $\beta\in\I$, then $\bar{e}_\alpha\bar{e}_\beta=j(e_\alpha e_\beta)\in j(J)$ and
\[ \bar{e}_\alpha\bar{e}_\beta - \sgn(\alpha,\beta)\sum_{\gamma\in\alpha\join\beta} \bar{e}_\gamma
= j\left( e_\alpha e_\beta-\sgn(\alpha,\beta)\sum_{\gamma\in(\alpha\join\beta)\cap\I} {e}_\gamma \right) - \sgn(\alpha,\beta)\sum_{\gamma\in(\alpha\join\beta)\setminus\I} \bar{e}_\gamma
\in j(J)+\bar{J}_0.\]
Now consider $\omega\in S\setminus\I$.
If $\omega$ is not unicircular, then $\bar{e}_{\omega\setminus a}\in\bar{J}_0$ for all $a\in\at(\omega)$.
If $\omega$ is $\zeta$-unicircular, then by \Cref{lem:u-a-ind},
\[  \sum_{a\in\at(\omega)} \sgn(a,\omega\setminus a) \bar{e}_{\omega\setminus a} 
= j\left( \sum_{a\in\at(\zeta)} \sgn(a,\omega\setminus a)e_{\omega\setminus a} \right) + \sum_{a\in\at(\omega)\setminus \at(\zeta)} \sgn(a,\omega\setminus a)\bar{e}_{\omega\setminus a} 
\in j(J)+\bar{J}_0.\]
Therefore, $\bar{J}\subseteq j(J)+\bar{J}_0$. Since $\bar{J}_0\subseteq\bar{J}$, rearranging the above equalities yields $j(J)\subseteq\bar{J}+\bar{J}_0=\bar{J}$.
\end{proof}

The next proposition relates our algebra to the classical Orlik--Solomon algebra for a geometric semilattice (see \cite{OS,yuz}), proving that ours is a true generalization.

\begin{prop}[Classical OS Algebra]\label{prop:classicalOS}
Let $\P$ be a geometric semilattice with a fixed linear order on $\at(\P)$. Then the graded algebra $A(\P)$ is isomorphic to the quotient $\tilde{E}/\tilde{J}$, where $\tilde{E}$ is the exterior algebra generated in degree 1 by $\tilde{e}_a$ for $a\in\at(\P)$, and $\tilde{J}\subseteq\tilde{E}$ is the ideal generated by
\begin{enumerate}
\item[($\tilde{\imath}$)] $\tilde{e}_{a_1}\cdots \tilde{e}_{a_k}$ whenever $a_{1}\vee\cdots\vee a_{k}=\varnothing$, and
\item[($\widetilde{\imath\imath\imath}$)] $\sum_{j=1}^k (-1)^{j-1} \tilde{e}_{a_1}\cdots\hat{\tilde{e}}_{a_j}\cdots\tilde{e}_{a_k}$ whenever $a_{1}\vee\cdots\vee a_{k}\not\subseteq\I$ and $a_1\ato\cdots\ato a_k$.
\end{enumerate}
The isomorphism is induced by the inclusion $\tilde{E}\into E$ with $\tilde{e}_a\mapsto e_{a}$ for $a\in\at(\P)$.
\end{prop}
\begin{proof}
Consider the presentation $A(\P)\cong\bar{E}/\bar{J}$ from \Cref{prop:altOS} and let 
\[\bar{J}_m:=\langle \bar{e}_{\alpha}-\bar{e}_{a_{1}}\cdots\bar{e}_{a_{k}}\st \at(\alpha) = \{a_{1}<\cdots<a_{k}\}\rangle\subseteq\bar{E}.\]
The inclusion $\iota\colon \tilde{E}\to\bar{E}$ induces an isomorphism $\tilde{E}\cong\bar{E}/\bar{J}_m$. We claim that $\bar{J}=\iota(\tilde{J})+\bar{J}_m$.

Since $\P$ is a geometric semilattice, so is $S$, thus every element $\alpha\in S$ is determined by its set of atoms, explicitly   $\bigvee\at(\alpha)=\{\alpha\}$.
This implies that $\bar{J}_m$ contains the elements of type $(\bar{\imath})$, and the elements of type $(\ol{\imath\imath})$ when $\alpha\vee\beta\neq\varnothing$.
Now if $\at(\alpha)=\{a_{1}<\dots<a_{k}\}$ and $\at(\beta)=\{b_1<\dots<b_l\}$, then
$\iota(\tilde{e}_{a_1}\cdots\tilde{e}_{a_k}\tilde{e}_{b_1}\cdots\tilde{e}_{b_l}) 
= \bar{e}_{\alpha}\bar{e}_\beta$  modulo $\bar{J}_m$. Moreover, if $\alpha\notin\I$, then $\sum_a\sgn(a,\alpha\setminus a)\bar{e}_{\alpha\setminus a}\in\iota(\tilde{J})+\bar{J}_m$, and 
$\bar{e}_\alpha\in \iota(\tilde{J})+\bar{J}_m$
since
 $\tilde{e}_{a_1}\cdots\tilde{e}_{a_k} = \tilde{e}_{a_1}\left( \sum_{j=1}^k(-1)^{j-1} \tilde{e}_{a_1}\cdots\hat{\tilde{e}}_{a_j}\cdots\tilde{e}_{a_k} \right)\in\tilde{J}$. 
\end{proof}

\begin{example}
Consider the example in \Cref{fig:P1}, which is a geometric semilattice.
In the presentation of \Cref{defn:OS}, there are 9 generators and 46 relations. 
There is, however, a more efficient presentation with 4 generators and 3 relations: the generators $\tilde{e}_i$ correspond to the atoms; two relations arise from $1\vee4=\varnothing$ and $2\vee3\vee4=\varnothing$; one relation $\tilde{e}_2\tilde{e}_3-\tilde{e}_1\tilde{e}_3+\tilde{e}_1\tilde{e}_2$ arises from the circuit. 

In fact, for a geometric semilattice,
the circuits are enough to generate the relations $(\widetilde{\imath\imath\imath})$. This is not, in general, true in our context.
For example, the ideal in \Cref{ex:OS} requires consideration for the unicircular elements. Additionally, the algebra in \Cref{ex:OS} cannot be generated in degree 1, unlike the geometric semilattice case.
\end{example}

\subsection{A Gr\"obner basis for the Orlik--Solomon ideal}\label{sec:gb}
There is a well-known combinatorial basis for the  Orlik--Solomon algebra of a geometric lattice, called the $\nbc$ basis, corresponding to a Gr\"obner basis in the Orlik--Solomon ideal (\cite[\S2.5]{yuz}). 
Our goal is to extend this to locally geometric posets, first requiring some generalities on the theory of Gr\"obner bases in graded-commutative algebras.

Let $F$ be a free graded-commutative $\Q$-algebra with generators $f_1,\dots,f_\ell$. As a vector space, $F$ has a basis $\B(F)$ given by monomials
\[ f^{\ol{m}} := f_1^{m_1}\cdots f_\ell^{m_\ell} \qquad 
\text{where } (m_1,\dots,m_\ell)\in\Z_{\geq0}^\ell \text{ and } m_i<2 \text{ whenever } \deg(f_i) \text{ is odd}. \] 
The degree of the monomial $f^{\ol{m}}$ is $\deg(f^{\ol{m}})=\sum_{i=1}^\ell m_i\deg(f_i)$.
Equip the monomial basis with a graded-lexicographic ordering, meaning that $f^{\ol{m}}\mo f^{\ol{n}}$ if either
\begin{itemize}
\item $\deg(f^{\ol m})<\deg(f^{\ol n})$, or 
\item $\deg(f^{\ol m})=\deg(f^{\ol n})$ and the leftmost nonzero entry of $\ol{n}-\ol{m}$ is positive. 
\end{itemize}
For a nonzero polynomial $g=\sum c(\ol{m})f^{\ol m}$, its \textbf{leading monomial} $\lm(g)$ is the largest monomial $f^{\ol m}$ in the support of $g$, and its \textbf{leading term} is $\lt(g):=c(\ol{m}) f^{\ol m}$ where $\lm(g)=f^{\ol m}$. 

\begin{defn}
A finite set $G$ of polynomials generating an ideal $J\subseteq F$ is a \textbf{Gr\"obner basis} of $J$ if
\[  \langle \lm(g) \st g\in G\rangle = \lm(J),
\]
where $\lm(J) := \langle \lm(g) \st g\in J\setminus \{0\} \rangle$.
\end{defn}

Following the perspective of \cite[Chapter II.4]{green}, we will study Gr\"obner bases in a free graded-commutative algebra by lifting to a $\Theta$-algebra $\hat{F}$. From our free graded-commutative $\Q$-algebra $F$, let $\hat{F}$ be the graded associative $\Q$-algebra generated by $\hat{f}_1,\dots,\hat{f}_\ell$ with $d_i=\deg(\hat{f}_i)=\deg(f_i)$, modulo the relations
\[ \hat{f}_i\hat{f}_j-(-1)^{d_id_j} \hat{f}_j\hat{f}_i 
\qquad \text{ for } i\neq j.\]
Let $q\colon \hat{F}\to F$ be the quotient map given by $\hat{f}_i\mapsto f_i$, so that $F\cong \hat{F}/\ker(q)$ where $\ker(q)$ is generated as an ideal by $\hat{f}_i^2$ for $\deg(f_i)$ odd. As a vector space, $\hat{F}$ has basis $\B(\hat{F})$ given by monomials
\[  \hat{f}^{\ol m} := \hat{f}_1^{m_1}\cdots \hat{f}_\ell^{m_\ell} \qquad \text{where } (m_1,\dots,m_\ell)\in\Z_{\geq0}^\ell \]
and $q$ has a section $s\colon F\to\hat{F}$ given by $s(f^{\ol m})=\hat{f}^{\ol m}$. Abbreviate $\hat{g}:=s(g)\in\hat{F}$ for the \emph{standard lift} of an element $g\in F$.
Equip $\B(\hat{F})$ with the analogous $\mo$ monomial ordering to define leading monomials and leading terms of polynomials in $\hat{F}$, and note that $q(\lm(\hat{g}))={\lm(g)}$ for $g\in F$.

Fix a finite set $\hat{G}\subseteq\hat{F}$. For any two polynomials $p_1,p_2\in\hat{G}$, define the \emph{S-polynomial} $\sigma(p_1,p_2)$ by
\[\sigma(p_1,p_2):= p_1\tau_1-p_2\tau_2 \]
where $\tau_1$ and $\tau_2$ satisfy
$\lt(p_1)\tau_1 = \lcm(\lm(p_1),\lm(p_2))=\lt(p_2)\tau_2$.
Say that $\sigma=\sigma(p_1,p_2)$ is \emph{weakly reducible}, written $\sigma\to0$, if either $\sigma=0$ or $\sigma$ is a linear combination of elements $p\in\hat{G}$ satisfying $\lm(p)\leq_{\mathrm{grlex}}\lm(\sigma)$.
We now state the main technique to be used in establishing our Gr\"obner basis.

\begin{prop}\label{prop:liftgb}
Let $F$ be a free graded-commutative algebra with generators $f_1,\dots,f_\ell$ and $\mo$ order on the monomial basis $\B(F)$.
Let $J\subseteq F$ be an ideal generated by a finite set $G$, and set 
\[\hat{G}:=\{ \hat{g}\st g\in G \} \cup \{ \hat{f}_i^2 \st \deg(f_i) \text{ is odd} \} \subseteq\hat{F}. \]
If the S-polynomial $\sigma$ is weakly reducible for all pairs in $\hat{G}$ whose leading monomials are not relatively prime, then $G$ is a Gr\"obner basis of $J$.
\end{prop}
\begin{proof}
If the S-polynomial $\sigma$ is weakly reducible for all pairs in $\hat{G}$ whose leading monomials are not relatively prime, then $\hat{G}$ is a Gr\"obner basis of the ideal $\langle \hat{G}\rangle=q^{-1}(J)\subseteq\hat{F}$ by \cite[Thm 4.31]{green}. Then the following equalities hold
\begin{align*}
\lm(J) 
&= q(\lm(q^{-1}(J))) 
&& \text{since $q(\lm(\hat{f}))=\lm(f)$ for any $f\in J\setminus0$,}\\
&= q(\langle \lm(p) \colon p\in\hat{G}\rangle) 
&& \text{since $\hat{G}$ is a Gr\"obner basis}, \\
&= \langle q(\lm(p)) \colon p\in\hat{G} \rangle 
&& \text{since $q$ is surjective,}\\
&= \langle \lm(g) \colon g\in G \rangle
&& \text{since $q(\lm(\hat{g}))=\lm(g)$.}
\end{align*}
Although Green assumed throughout that the underlying field had odd characteristic, one only needs here that the characteristic is not equal to 2.
\end{proof}

The goal is now to apply this theory to the Orlik--Solomon ideal in the free graded-commutative algebra $E=E(\P)$. 
In fact, our Gr\"obner basis will hold with \emph{integer} coefficients; we start the proof of \Cref{thm:gb} explaining how to make the jump from $\Q$ to $\Z$. First, let us set a linear order on the generators of $E$ which will then determine the graded-lexicographic monomial ordering $\mo$.
Start with a fixed linear order $\ato$ on the atoms of $\I$, and let $\leq_*$ be a linear extension of $\I$ in which
$\alpha\leq_*\beta$ if either
\begin{itemize}
    \item $|\alpha|<|\beta|$, or
    \item $|\alpha|=|\beta|$ and the minimal atom in the symmetric difference $\at(\alpha)\triangle\at(\beta)$ is in $\at(\beta)$.
\end{itemize}
Then we can list the elements of $\I\setminus\{\zero\}$ in a chain $\alpha_1 <_*\cdots <_* \alpha_\ell$ and order monomials $e_{\alpha_1}^{m_1}\cdots e_{\alpha_\ell}^{m_\ell}$ in $E(\P)$ by $\mo$ accordingly, noting that if $|\alpha|=|\beta|$ and 
$\alpha <_*\beta$ then $e_{\beta}{\mo} e_\alpha$.  

\begin{example}\label{ex:grobner}
Consider the matroid prescheme shown in \Cref{fig:ex2} and \Cref{ex:OS}. One admissible choice of linear ordering on $\I\setminus\{\zero\}$ is
\[1 \io 2 \io 3 \io 4\]
\[(c,12)\io(c,13)\io(x,14)\io(c,23)\io(y,24)\io(z,34)\]
\[(u_1,124)\io(u_2,124)\io(u_1,134)\io(u_2,134)\io(u_1,234)\io(u_2,234)\]
where rank-1 elements necessarily precede rank-2 elements, which precede rank-3 elements.

Then the leading term in \eqref{eq:ex2:ii} is $e_4e_{(c,23)}$.
In \eqref{eq:ex2:iii}, the leading terms are $e_{(u_1,234)}$ and $e_{(u_2,234)}$, and in \eqref{eq:ex2:iiic}, the leading term is $e_{(c,23)}$. 

In this case, multiplying \eqref{eq:ex2:iiic} by $e_4$ and subtracting from \eqref{eq:ex2:ii} will cancel leading terms, which results in the S-polynomial
\begin{equation}\label{eq:spoly}
    -e_{(u_1,234)}-e_{(u_2,234)}+e_4e_{(c,13)}-e_4e_{(c,12)}.
\end{equation}
The latter two terms are the leading terms of their own respective type \eqref{OSii} generators, and the sum of these two generators subtracted by the two generators in \eqref{eq:ex2:iii} results in \eqref{eq:spoly}. 
This demonstrates the sort of reduction to be used in the proof of \Cref{thm:gb}.
\end{example}
Before delving into the main theorem of this section, we record here a quick lemma about signs of permutations that will be used repeatedly.

\begin{lemma}\label{lem:signs}
Let $X$ be a totally ordered set, and let $A$, $B$, and $C$ be pairwise disjoint subsets of $X$. Then 
\begin{enumerate}
\item \label{eq:comm}
$\sgn(A,B) = (-1)^{|A|\cdot|B|}\sgn(B,A)$, and
\item \label{eq:trans}
$\sgn(A,B)\sgn(B,C) = \sgn(A\cup B,C)\sgn(A,B\cup C)$.
\end{enumerate}
\end{lemma}
\begin{proof}
Recall that $\sgn(A,B) = (-1)^{\inv(A,B)}$ where $\inv(A,B)$ is the number of pairs $(a,b)\in A\times B$ with $a>b$.
Equation \eqref{eq:comm} follows from the observation that $|A|\cdot|B|=\inv(A,B)\inv(B,A)$.

Now
$\sgn(A,B\cup C) = \sgn(A,B)\sgn(A,C)$
 because $\inv(A,B\cup C) = \inv(A,B)+\inv(A,C)$, and similarly 
 $\sgn(A\cup B, C) = \sgn(A,C)\sgn(B,C)$. Combining these two equations proves \eqref{eq:trans}.
\end{proof}

\begin{theorem}[Gr\"obner basis for OS ideal]\label{thm:gb}
Let $\P$ be a locally geometric poset, with associated matroid prescheme $M(\P)=(S,I)$. Let $G\subseteq E(\P)$ denote the set: 
\begin{enumerate}[label=(\roman*)]
    \item $e_\alpha e_\beta$ when $\alpha\meet\beta\neq\{\hat{0}\}$, except when $\alpha=\beta$ and $|\alpha|$ is odd;
    \item $e_\alpha e_\beta-\sgn(\alpha,\beta)\sum\limits_{\gamma\in(\alpha\join\beta)\cap\I}e_\gamma$ when $\alpha\meet\beta=\{\hat{0}\}$; and
    \item $\sum\limits_{a\in\at(\zeta)}\sgn(a,\omega\setminus a)e_{\omega\setminus a}$ when $\omega$ is a $\zeta$-unicircular element.
    \end{enumerate}
Then $G$ is a Gr\"obner basis for the Orlik--Solomon ideal $J(\P)\subseteq E(\P)$.
\end{theorem}
\begin{proof}
Since the leading coefficients of all elements in $G$ are equal to 1, their leading terms will generate $\lt(J)$ over $\Z$ if and only if they do so over $\Q$. This allows us to  
use \Cref{prop:liftgb}, for which we work over $\Q$
in the $\Theta$-algebra $\hat{E}$  generated by $z_\alpha:=\hat{e}_\alpha$ for $\alpha\in\I$. Then $\hat{G}\subseteq\hat{E}$ contains the elements 
\begin{enumerate}
    \item[\mylabel{i'}{i'}] $z_\alpha z_\beta$ when $\alpha\meet\beta\neq\{\hat{0}\}$;
    \item[\mylabel{ii'}{ii'}] $z_\alpha z_\beta-\sgn(\alpha,\beta)\sum\limits_{\gamma\in(\alpha\join\beta)\cap\I}z_\gamma$ when $\alpha\meet\beta=\{\hat{0}\}$; and
    \item[\mylabel{iii'}{iii'}] $\sum\limits_{a\in\at(\zeta)}\sgn(a,\omega\setminus a)z_{\omega\setminus a}$ when $\omega$ is a $\zeta$-unicircular element,
    \end{enumerate}
noting that the elements $z_\alpha^2$ for $|\alpha|$ odd are included in \eqref{i'}.

The rest of the proof is spent reducing the S-polynomial $\sigma$ for all pairs in $\hat{G}$ whose leading monomials are not relatively prime. The computations are, as expected, quite technical.

\begin{description}
\item[Case (\ref{i'},\ref{i'})]  
    Any two such generators are monomials, thus $\sigma=0$. 
 
\item[Case (\ref{i'},\ref{ii'})] 
  Suppose that $\alpha_1\meet\beta_1\neq\{\zero\}$ and $\alpha_2\meet\beta_2=\{\zero\}$. Consider 
\[z_{\alpha_1}z_{\beta_1},\textup{ and}\quad z_{\alpha_2}z_{\beta_2}-\sgn(\alpha_2,\beta_2)\sum_{\gamma\in(\alpha_2\join\beta_2)\cap\I}z_\gamma.\]
  Without loss of generality, assume that $\alpha=\alpha_1=\alpha_2$. Then the S-polynomial is
  \[\sigma=(-1)^{|\beta_1|\cdot|\beta_2|}\sgn(\alpha,\beta_2)\sum_{\gamma\in(\alpha\join\beta_2)\cap\I}z_\gamma z_{\beta_1}.\]
  Since
  the assumption $\alpha\meet\beta_1\neq\{\zero\}$ implies
  $\gamma\meet\beta_1\neq\{\zero\}$ whenever $\gamma\in(\alpha\vee\beta_2)\cap\I$, each monomial in the support of $\sigma$ is of type \eqref{i'}. Thus $\sigma\to0$.

\item[Case (\ref{i'},\ref{iii'})]

  Suppose that $\alpha\meet\beta\neq\{\hat{0}\}$ and $\omega$ is a $\zeta$-unicircular element. Consider 
  \[z_{\alpha}z_{\beta},\textup{ and}\quad \sum\limits_{a\in\at(\zeta)}\sgn(a,\omega\setminus a)z_{\omega\setminus a}.\]
  Without loss of generality,  assume that $\alpha=\omega\setminus a_1$ where $a_1=\min_{\ato}\at(\zeta)$. Then 
  \[\sigma=-\sgn(a_1,\omega\setminus a_1)\sum_{a\in\at(\zeta)\setminus a_1}\sgn(a,\omega\setminus a)z_{\omega\setminus a}z_\beta.\]  
Consider $a\in\at(\zeta)\setminus a_1$. 
If $(\omega\setminus a)\meet\beta\neq\{\zero\}$, then the term $z_{\omega\setminus a}z_\beta$ is of type \eqref{i'}. 
Otherwise, if $(\omega\setminus a)\meet\beta=\zero$, then $((\omega\setminus a)\vee\beta)\cap\I=\varnothing$ by \Cref{lem:i'-iii'} and the term $z_{\omega\setminus a}z_\beta$ is of type \eqref{ii'}.
Since each monomial in the support of $\sigma$ is itself in $\hat{G}$, $\sigma\to0$.

\item[Case (\ref{ii'},\ref{ii'})]
Suppose, without loss of generality, that $\alpha\meet\beta_1=\{\zero\}=\alpha\meet\beta_2$. 
The S-polynomial of 
\[
z_{\alpha}z_{\beta_1}-\sgn(\alpha,\beta_1)\sum\limits_{\gamma_1\in(\alpha\join\beta_1)\cap\I}z_{\gamma_1}\textup{ and}\quad z_{\alpha}z_{\beta_2}-\sgn(\alpha,\beta_2)\sum\limits_{\gamma_2\in(\alpha\join\beta_2)\cap\I}z_{\gamma_2}
\]
is
  \[\sigma=-\sgn(\alpha,\beta_1)\sum_{\gamma_1\in(\alpha\join\beta_1)\cap\I}z_{\gamma_1}z_{\beta_2}+(-1)^{|\beta_1|\cdot|\beta_2|}\sgn(\alpha,\beta_2)\sum_{\gamma_2\in(\alpha\join\beta_2)\cap\I}z_{\gamma_2}z_{\beta_1}.\]
If $\beta_1\meet\beta_2\neq\{\zero\}$, then every monomial in the support of $\sigma$ is of type \eqref{i'}, thus $\sigma\to0$.
If $\beta_1\meet\beta_2=\{\zero\}$, then every monomial in the support of $\sigma$ is the leading monomial of a type \eqref{ii'} element. 
In this case, $\sigma$ reduces to
  \[\begin{split}
  \overline{\sigma} = & -\sgn(\alpha,\beta_1)\sgn(\alpha\join\beta_1,\beta_2)\sum_{\gamma_1\in(\alpha\join\beta_1)\cap\I}\sum_{\delta\in(\gamma_1\join\beta_2)\cap\I}z_\delta\\
  & +(-1)^{|\beta_1|\cdot|\beta_2|}\sgn(\alpha,\beta_2)\sgn(\beta_2\join\alpha,\beta_1)\sum_{\gamma_2\in(\alpha\join\beta_2)\cap\I}\sum_{\delta\in(\gamma_2\join\beta_1)\cap\I}z_\delta.
  \end{split}\] 
Now the summations can be reindexed using, for $\{i,j\}=\{1,2\}$,
  \[\bigcup_{\gamma_i\in(\alpha\join\beta_i)\cap\I}(\gamma_i\join\beta_j)\cap\I
  =\left( \bigcup_{\gamma_i\in\alpha\join\beta_i} \gamma_i\join\beta_j\right)\cap\I
  =\left(\bigvee\{\alpha,\beta_1,\beta_2\}\right)\cap\I\]
via \Cref{lem:triplejoin}, hence
  \[\begin{split}
  \overline{\sigma} = & -\sgn(\alpha,\beta_1)\sgn(\alpha\join\beta_1,\beta_2)\sum_{\bigvee\{\alpha,\beta_1,\beta_2\}\cap\I}z_\gamma\\
  & +(-1)^{|\beta_1|\cdot|\beta_2|}\sgn(\alpha,\beta_2)\sgn(\beta_2\join\alpha,\beta_1)\sum_{\gamma\in\bigvee\{\alpha,\beta_1,\beta_2\}\cap\I}z_\gamma.
  \end{split}\]
  But then $\ol{\sigma}=0$ because, combining  \Cref{lem:signs}\eqref{eq:comm} and \Cref{lem:signs}\eqref{eq:trans},
    \[\sgn(\alpha,\beta_1)\sgn(\alpha\join\beta_1,\beta_2)=(-1)^{|\beta_1|\cdot|\beta_2|}\sgn(\alpha,\beta_2)\sgn(\beta_2\join\alpha,\beta_1).\]

\item[Case (\ref{ii'},\ref{iii'})]
  Suppose that $\alpha\meet\beta=\{\hat{0}\}$, and let $\omega$ be a $\zeta$-unicircular element. Consider 
  \[z_{\alpha}z_{\beta}-\sgn(\alpha,\beta)\sum\limits_{\gamma\in(\alpha\join\beta)\cap\I}z_{\gamma},\textup{ and}\quad \sum\limits_{a\in\at(\zeta)}\sgn(a,\omega\setminus a)z_{\omega\setminus a}.\]
  Without loss of generality, assume  $\alpha=\omega\setminus a_1$ where $a_1=\min_{\ato}\at(\zeta)$. The S-polynomial is
  \[\sigma=-\sgn(\omega\setminus a_1,\beta)\sum_{\gamma\in((\omega\setminus a_1)\join\beta)\cap\I}z_\gamma-\sgn(a_1,\omega\setminus a_1)\sum_{a\in\at(\zeta)\setminus a_1}\sgn(a,\omega\setminus a)z_{\omega\setminus a}z_\beta.\]
If $a_1\leq \beta$, then the first sum vanishes by \Cref{lem:i'-iii'}, and every monomial in the second sum is of type \eqref{i'}, thus $\sigma\to0$.
  If $a_1\not\leq\beta$, then $\omega\meet\beta=\{\zero\}$. In this case, every monomial in the second sum is the leading monomial of a type \eqref{ii'} element, thus $\sigma$ reduces to
  
 \begin{align*}
 \overline{\sigma}&= -\sgn(a_1,\omega\setminus a_1)\sum_{a\in\at(\zeta)}\sgn(a,\omega\setminus a)\sgn(\omega\setminus a,\beta)\sum_{\gamma\in((\omega\setminus a)\join\beta)\cap\I}z_\gamma \\
 &= -\sgn(a_1,\omega\setminus a_1)\sum_{\substack{\upsilon\in\omega\join\beta\\ \zeta\text{-unicirc.}}}\sum_{a\in\at(\zeta)}\sgn(a,\omega\setminus a)\sgn(\omega\setminus a,\beta)z_{\upsilon\setminus a}
 && \text{by \Cref{lem:ii'-iii'}} \\
 &= -\sgn(a_1,\omega\setminus a_1)\sum_{\substack{\upsilon\in\omega\join\beta\\ \zeta\text{-unicirc.}}}\sum_{a\in\at(\zeta)}\sgn(a,(\omega\setminus a)\join\beta)\sgn(\omega,\beta)z_{\upsilon\setminus a}
 &&\text{by \Cref{lem:signs}\eqref{eq:trans}}\\
 &= -\sgn(a_1,\omega\setminus a_1)\sgn(\omega,\beta)\sum_{\substack{\upsilon\in\omega\join\beta\\ \zeta\text{-unicirc.}}}\left(\sum_{a\in\at(\zeta)}\sgn(a,\upsilon\setminus a)z_{\upsilon\setminus a}\right).
 \end{align*}
Then $\ol{\sigma}\to0$ because this is a sum of type \eqref{iii'} elements, and therefore $\sigma\to0$.

\item[Case (\ref{iii'},\ref{iii'})]
  Let $\omega_i$ be $\zeta_i$-unicircular 
  for $i=1,2$, with $\omega_1\neq\omega_2$, and consider
  \[\sum\limits_{a\in\at(\zeta_1)}\sgn(a,\omega_1\setminus a)z_{\omega_1\setminus a},\textup{ and}\quad \sum\limits_{a\in\at(\zeta_2)}\sgn(a,\omega_2\setminus a)z_{\omega_2\setminus a}.\]
  Without loss of generality,  assume  $\omega_1\setminus a_1=\omega_2\setminus a_2$ where $a_i=\min_{\ato}(\at(\zeta_i))$ and  $a_1\ato a_2$. 
  Then the S-polynomial  is 
  \[\begin{split}
  \sigma= & \sgn(a_1,\omega_1\setminus a_1)\sum_{a\in\at(\zeta_1)\setminus a_1}\sgn(a,\omega_1\setminus a)z_{\omega_1\setminus a}\\
   & -\sgn(a_2,\omega_2\setminus a_2)\sum_{a\in\at(\zeta_2)\setminus a_2}\sgn(a,\omega_2\setminus a)z_{\omega_2\setminus a}.
  \end{split}\] 
  Using \Cref{lem:iii'-iii'}, fix $\upsilon\in\omega_1\vee\omega_2$ so that, for each $a\in\at(\zeta_1)\cup\at(\zeta_2)$, $\upsilon\setminus a$ is $\zeta^\upsilon_a$-circular.
For $a,b\in\at(\zeta_1)\cup\at(\zeta_2)$, let 
\[s_{b,a} = \sgn(b,\upsilon\setminus b)\sgn(a,(\upsilon\setminus b)\setminus a),\]
noting $s_{b,a}=\sgn(a,b)\sgn(a\vee b,(\upsilon\setminus a)\setminus b) = -s_{a,b}$ by \Cref{lem:signs},
and then consider 
\[p = \sum_{b\in\at(\zeta_1)\cup\at(\zeta_2)} 
\ \sum_{a\in\at(\zeta^\upsilon_b)} s_{b,a} z_{(\upsilon\setminus b)\setminus a}. \]
By \Cref{lem:iii'-iii'}, the monomial $z_{(\upsilon\setminus b)\setminus a} = z_{(\upsilon\setminus a)\setminus b}$ appears in $p$ with coefficient $s_{b,a}+s_{a,b}=0$. Thus $p=0$.
Then, with an appropriate sign, the terms in $p$ indexed by $b=a_1$ and $b=a_2$ cancel with terms in $\sigma$, providing
\begin{align*}
\sigma = \sigma+s_{a_1,a_2}p
&= \sum_{a\in\at(\zeta_1)\setminus a_1} s_{a_2,a_1}s_{a_2,a}z_{\omega_1\setminus a} - \sum_{a\in\at(\zeta_2)\setminus a_2} s_{a_1,a_2}s_{a_1,a}z_{\omega_2\setminus a} \\
&\qquad \qquad + \sum_{b\in\at(\zeta_1)\cup\at(\zeta_2)} 
\ \sum_{a\in\at(\zeta^\upsilon_b)} s_{a_1,a_2}s_{b,a} z_{(\upsilon\setminus b)\setminus a}\\
&= \sum_{b\in\at(\zeta_1)\setminus a_1\cup\at(\zeta_2)\setminus a_2} s_{a_1,a_2}\sgn(b,\upsilon\setminus b) \sum_{a\in\at(\zeta^\upsilon_b)} \sgn(a,(\upsilon\setminus b)\setminus a) z_{(\upsilon\setminus b)\setminus a}.
\end{align*}
As each inner sum is a type \eqref{iii'} element, for  $\upsilon\setminus b$ unicircular and with leading monomial preceding that of $\sigma$,  $\sigma\to0$.

\end{description}
\end{proof}

\subsection{The no-broken-circuit basis}\label{sec:nbcbasis}
We now use our Gr\"obner basis to obtain a combinatorial $\nbc$ basis for the Orlik--Solomon algebra, as a generalization of that for geometric lattices due to \cite{JT}. 
The $\nbc$ basis is then used to prove several interesting properties of the Orlik--Solomon algebra. We start by establishing the applicable tool from Gr\"obner basis theory. 
Although this is a statement with rational coefficients, we will be able to utilize it in our case to obtain a $\Z$-basis (as in \Cref{cor:vsbasis}).

\begin{prop}\label{prop:vsbasis}
Let $F$ be a finitely-generated free graded-commutative algebra over $\Q$, equipped with the $\mo$ ordering on its monomial basis $\mathcal{B}(F)$, and let $I\subseteq F$ be an ideal. Then $F/I$ has a $\Q$-vector space basis represented by monomials in $\mathcal{B}(F)$ which are not contained in $\lm(I)$.
\end{prop}
\begin{proof}
First, let $\mathcal{B}(I)=\mathcal{B}(F)\cap\lm(I)$, which is a monomial basis for $\lm(I)$. 
For each $u\in\mathcal{B}(I)$, pick $g_u\in I$ such that $u=\lt(g_u)$. Let $G=\{g_u\st u\in\mathcal{B}(I)\}$.
We will first show that $G$ is a $\Q$-basis for $I$. 

To see that $G$ spans $I$, let $f\in I\setminus\{0\}$. 
Then $\lt(f)=a_uu$ where $u=\lm(f)\in\B(I)$ and $a_u\in\Q$.
Consider $f'=f-a_ug_u$, so that $f'\in I$ and $\lm(f')\mo\lm(f)$. Induction on the order of leading monomials completes the proof that $G$ spans $I$. For independence, suppose that $c_1g_{u_1}+\cdots +c_ng_{u_n}=0$ for some scalars $c_1,\dots,c_n$ and monomials $u_1,\dots,u_n\in\mathcal{B}(I)$ with $u_1\mo\cdots\mo u_n$.
When expanding the expression $c_1g_{u_1}+\cdots+c_ng_{u_n}$ in the monomial basis, the coefficient of $u_i$ is $c_1+\cdots+c_i$ for each $i$. By independence of monomials, this implies each $c_i=0$.

By a similar argument,  $G\cup(\mathcal{B}(F)\setminus\mathcal{B}(I))$ is a $\Q$-basis for $F$. Then the quotient map $F\to F/I$ restricts to a vector space isomorphism on the subspace of $F$ spanned by $\mathcal{B}(F)\setminus\mathcal{B}(I)$, that is, monomials which are not contained in $\lm(I)$. In this way, such monomials determine a basis for the quotient $F/I$.
\end{proof}

\begin{theorem}[$\nbc$ basis]\label{cor:vsbasis}
The Orlik--Solomon algebra $A=A(\P)$ of a locally geometric poset $\P$ 
is a free $\Z$-module with $\Z$-basis given by
        \[\B(A)=\{e_\alpha \colon \alpha\in \nbc(\P)\}.\]
\end{theorem}
\begin{proof}
Using the Gr\"obner basis $G$ from \Cref{thm:gb},
\[ \lm(J) = \langle \lm(g)\st g\in G\rangle =  \langle e_{\alpha}e_{\beta} \st \alpha,\beta\in\I\rangle + \langle e_{\omega\setminus a} \st \omega \text{ is $\zeta$-unicircular and } a=\min\! _{\ato}\at(\zeta) \rangle. \]
Then, recalling the convention $e_{\zero}=1$, the monomials not in $\lm(J)$ are precisely the monomials in $\B(A)$ by the characterization of $\nbc$ in \Cref{lem:nbc-uni}. 
Since the leading coefficients of elements in $G$ are all equal to 1, we have $\lt(J)=\lm(J)$. 
This implies that 
the $\Q$-basis in the proof of \Cref{prop:vsbasis} is in fact a $\Z$-basis, and our result follows.
\end{proof}

Note that the defining relations in \Cref{defn:OS} are homogeneous, so that $A(\P)$ is graded. Using the combinatorial basis, we can refine this grading to one indexed by the poset $\P$ and obtain a combinatorial formula for the Hilbert series, generalizing the geometric lattice case in \cite[Proposition 2.10, Theorem 2.6]{OS}.
It is worth pointing out that, although $A(\P)$ is a $\P$-graded algebra when $\P$ is a geometric lattice, our $\P$-grading  does not, in general, respect the multiplication in this way.
To make sense of the Hilbert series formula, recall the recursively-defined M\"obius function $\mu:\P\to\Z$ given by $\mu(\zero)=1$ and $\sum_{y\leq x}\mu(y)=0$ for any $x>\zero$, and that the \emph{characteristic polynomial} of $\P$ is defined as $\chi_\P(t) = \sum_{x\in\P} \mu(x)t^{\rk(\P)-\rk(x)}$.

\begin{theorem}[$\P$-decomposition of $A(\P)$]\label{cor:brieskorn}
Let $\P$ be a locally geometric poset with rank function $\rk$. Then the Orlik--Solomon algebra admits the following  decomposition of free $\Z$-modules
\[A(\P) 
\cong \bigoplus_{x\in\P} A^{\rk(x)}(\P_{\leq x}) \]
 with the degree-$j$ component of $A(\P)=\bigoplus_j A^j(\P)$ given by
\[A^j(\P)\cong \bigoplus_{\substack{x\in\P\\ \rk(x)=j}} A^{j}(\P_{\leq x})\]
The Hilbert series of $A(\P)$ is 
\[ H_\P(t):=\sum_{j\geq0} \dim(A^j)t^j = \sum_{x\in\P} \mu(x)(-t)^{\rk(x)} = (-t)^{\rk(\P)}\chi_\P\left(-\frac{1}{t}\right) \]
\end{theorem}
\begin{proof}
For $x\in\P$, let $A_x\subseteq A(\P)$ be the submodule spanned by $e_{(x,T)}$ with $(x,T)\in \nbc(\P)$, or equivalently with $(x,T)\in\nbc(\P_{\leq x})$ by \Cref{lem:localnbc}. 
Then 
$A_x\cong A^{\rk(x)}(\P_{\leq x})$  
since $\P_{\leq x}$ is a geometric lattice (see \cite{JT,yuz}).
Writing the $\nbc$ basis of \Cref{cor:vsbasis} as a disjoint union
$ \bigsqcup \{e_{(x,T)} \st (x,T)\in \nbc(\P)\} $ indexed by $x\in\P$ yields the decomposition  $A=\bigoplus_{x\in\P}A_x$.
It is clear that $A_x$ is homogeneous of degree $\rk(x)$.
The Hilbert series formula follows from the decomposition because $\dim(A^{\rk(x)}(\P_{\leq x}))=(-1)^{\rk(x)}\mu(x)$ for any $x\in\P$ by \cite[Theorem 2.6]{OS}.
\end{proof}

Now given a locally geometric poset $\P$, consider the deletion and contraction of its associated matroid prescheme $M=M(\P)$ as defined in \Cref{sec:del-con}. Given the triple of matroid preschemes, $(M, M', M'')$, we denote the poset of flats of $M'$ as $\P'$ and of $M''$ as $\P''$. We now compare the corresponding Orlik--Solomon algebras, denoted as  
$A':=E(\P')/J(\P')$ and $A'':=E(\P'')/J(\P'')$. 
The following theorem uses the $\nbc$ basis to establish a relationship these three algebras.
In the case of geometric semilattices, \cite[Thm 3.8]{JT} obtained the same formula using the $\nbc$ basis.
Although we do not employ it here, this deletion and contraction formula can be used as a recursive tool.

\begin{theorem}\label{thm:del-con}
    Let $\P$ be a locally geometric poset with associated matroid prescheme $M=M(\P)$ and triple  $(M,M',M'')$. Then $A'$ is a subalgebra of $A$, and
    there exists a short exact sequence of $\Z$-modules
    \[0\to A'\to A\to A''\to 0.\]
\end{theorem}
\begin{proof}
From \Cref{cor:vsbasis}, $A$ has $\Z$-basis given by $\B(A)=\{e_\alpha\st\alpha\in\nbc(M)\}$. 
By \Cref{lem:nbc}, $\B(A)$ is a disjoint union of $\B'=\{e_\alpha\st \alpha\in\nbc(M')\}$ and $\B''=\{e_\alpha\st \alpha\in\nbc(M'')\}$.
As these may be identified with bases for $A'$ and $A''$, respectively, by \Cref{cor:vsbasis}, there is a decomposition $A\cong A'\oplus A''$ hence short exact sequence as claimed.

To see that the injection $A'\to A$, induced by the inclusion $E'\subseteq E$, is an algebra map, we show $J'\subseteq J$.
Since $\I'=\I\cap S_{\not\geq a}$ is downward-closed, an element in $J'$ of type \eqref{OSi} must also be in $J$. 
A type \eqref{OSii} element of $J'$ must be in $J$ because $\alpha\vee_S\beta = \alpha\vee_{S\not\geq a}\beta$ for all $\alpha,\beta\in S_{\not\geq a}$.
Type \eqref{OSiii} elements of $J'$ are in $J$ because, if $\omega$ is $\zeta$-unicircular in $M'$, then $\rk(\omega)=\rk'(\omega)=|\omega|-1$ and hence $\omega$ is $\zeta$-unicircular in $M$ by \Cref{prop:rank-uni}.
\end{proof}

\subsection{The Orlik--Solomon sheaf}\label{sec:sheaf}

In this section, we describe how the Orlik--Solomon algebra defined in \Cref{defn:OS} naturally arises from the Orlik--Solomon algebra for geometric lattices. For any element $x$ in a locally geometric poset $\P$, the poset $\P_{\leq x}$ is a geometric lattice with its associated Orlik--Solomon algebra $A(\P_{\leq x})$.
Moreover, if $x\geq y$, then there is a projection $r^x_y\colon A(\P_{\leq x})\to A(P_{\leq y})$. 
Viewing $\P$ as a category whose objects are the elements of $\P$ and whose morphisms $x\to y$ correspond to order relations $x\geq y$, this data determines a functor $\F:\P\to\mathtt{Alg}_\Z$ to the category of $\Z$-algebras. Such a functor is called a \emph{sheaf} on the poset $\P$.
Endowing $\P$ with the Alexandroff topology, in which the intervals $\P_{\leq x}$ form a basis, this agrees with the usual notion of sheaf on a topological space.
In fact, the open sets are \emph{order ideals}, that is, downward-closed subsets of $\P$, and the restriction maps in the sheaf are given as follows.

\begin{prop}\label{prop:OSmaps}
The inclusion $i\colon\mathcal{Q}\into \P$ of an order ideal induces a surjective algebra homomorphism $i^*\colon A(\P)\to A(\mathcal{Q})$.
Moreover, $(i\circ j)^*=j^*\circ i^*$ whenever $j\colon\mathcal{Q}'\into \mathcal{Q}$ is the inclusion of an order ideal. 
\end{prop}
\begin{proof}
As $\mathcal{Q}$ is an order ideal of $\P$,
we have that
$S(\mathcal{Q})=\{(x,T)\in S(\P)\st x\in \mathcal{Q}\}$ and similarly $\I(\mathcal{Q})=\{(x,T)\in \I(\P)\st x\in \mathcal{Q}\}$. 
Letting $e_\alpha$ denote the generators of $A(\P)$, where $\alpha\in\I(\P)$, and $e'_\beta$ denote the generators of $A(\mathcal{Q})$, where $\beta\in\I(\mathcal{Q})$, we claim that the following defines a surjective algebra homomorphism. 
\[
i^*(e_\alpha) =
\begin{cases}
e'_\alpha & \text{ if } \alpha\in\I(\mathcal{Q}) \\
0 & \text{ otherwise}
\end{cases}
\]
To see that $i^*$ respects relation \eqref{OSi}, note that if $\alpha,\beta\in\I(\mathcal{Q})$, then $\alpha\wedge_{S(\P)}\beta\neq\{\zero\}$ implies $\alpha\wedge_{S(\mathcal{Q})}\beta\neq\{\zero\}$.
To see that $i^*$ respects relation \eqref{OSii}, note that if $\alpha,\beta\in\I(\mathcal{Q})$, then $\alpha\vee_{S(\mathcal{Q})}\beta\subseteq\alpha\vee_{S(\P)}\beta$, and otherwise $(\alpha\vee_{S(\P)}\beta)\cap S(\mathcal{Q})=\varnothing$.
To see that $i^*$ respects relation \eqref{OSiii}, note that for a $\zeta$-unicircular element $\omega=(x,T)$ in $S(\P)$, if $x\in\mathcal{Q}$ then $\omega$ is $\zeta$-unicircular in $S(\mathcal{Q})$, and if $x\notin \mathcal{Q}$ then $\omega\setminus a=(x,T\setminus a)\notin S(\mathcal{Q})$ for any $a$.

This is compatible with compositions because any order ideal $\mathcal{Q}'$ of $\mathcal{Q}$ is also an order ideal of $\P$, with 
$S(\mathcal{Q}') =  \{(x,T)\in S(\mathcal{Q}) \st x\in\mathcal{Q}'\} = \{(x,T)\in S(\P)\st x\in\mathcal{Q}'\}$ and similarly for $\I(\mathcal{Q}')$.
\end{proof}

The following theorem shows how, for a locally geometric poset $\P$, the \emph{local} Orlik--Solomon algebras $A(\P_{\leq x})$ glue together to the \emph{global} Orlik--Solomon algebra $A(\P)$. In the case of geometric semilattices, this idea was first considered by Yuzvinsky \cite{yuz-sheaf} (see also \cite{BDF}).

\begin{theorem}[Local-to-global OS algebra]\label{thm:sheaf}
The sheaf $\F(\P)$ of Orlik--Solomon algebras on a locally geometric poset $\P$ is a flasque sheaf with global sections isomorphic to the algebra  $A(\P)$.
\end{theorem}
\begin{proof}
Let \[\Gamma=\left\{ (a_x)_{x\in\P} \in \prod A(\P_{\leq x}) \st r^x_y(a_x)=a_y \text{ whenever } y\leq x\right\}\] be the ring of global sections of $\F(\P)$. 
Using, for each $x\in \P$, the projection $r_x\colon A(\P)\to A(\P_{\leq x})$ from \Cref{prop:OSmaps} and the fact $r^x_y\circ r_x=r_y$ whenever $x\geq y$, we have an algebra map
\[ \phi\colon A(\P)\to\Gamma, \qquad \phi(a)= (r_x(a))_{x\in\P} \]

To show that $\phi$ is an isomorphism, we use the $\P$-decomposition $A(\P)\cong \bigoplus_{x\in\P} A^{\rk(x)}(\P_{\leq x})$ from \Cref{cor:brieskorn}. For $x\in\P$, let $\pi_x\colon A(\P)\to A^{\rk(x)}(\P_{\leq x})$ be the projection onto the component indexed by $x$, so that we may identify an element $a\in A(\P)$ with the tuple $(\pi_x(a))_{x\in\P}\in\bigoplus_{x\in\P}A^{\rk(x)}(\P_{\leq x})$.
Using the analogous decomposition of $A(\P_{\leq x})$ and projections $\pi^x_y\colon A(\P_{\leq x})\to A^{\rk(y)}(\P_{\leq y})$ for $y\leq x$,
we may similarly identify an element $a_x\in A(\P_{\leq x})$ with the tuple $(\pi^x_y(a_x))_{y\leq x}\in \bigoplus_{y\leq x} A^{\rk(y)}(\P_{\leq y})$. 
Observe that $\pi^x_y=\pi^y_y\circ r^x_y$ for $x\geq y$, and 
$r_x(a)=(\pi_y(a))_{y\leq x}$ for $a\in A(\P)$ and $x\in\P$. 

Now, for injectivity of $\phi$, suppose $\phi(a)=\phi(b)$. Then $(\pi_y(a))_{y\leq x}=r_x(a)=r_x(b)=(\pi_y(b))_{y\leq x}$ for all $x\in\P$, which implies $\pi_y(a)=\pi_y(b)$ for all $y\in\P$, hence $a=b$.

For surjectivity, suppose $(a_x)_{x\in\P}\in\Gamma$, and let $a=(\pi^x_x(a_x))_{x\in\P}\in A(\P)$. Then 
\[ r_x(a) = (\pi_y(a))_{y\leq x} = (\pi^y_y(a_y))_{y\leq x} = (\pi^y_y(r^x_y(a_x)))_{y\leq x} = (\pi^x_y(a_x))_{y\leq x} =a_x \]
for all $x\in\P$, implying $\phi(a)=(a_x)_{x\in\P}$.

Finally, the sheaf $\F(\P)$ is flasque because all restriction maps are surjective via \Cref{prop:OSmaps}.
\end{proof}

\section{Abelian arrangements and their cohomology}\label{sec:H^*}

\subsection{Abelian arrangements} \label{sec:arrangement}
This section reviews a topological manifestation of geometric posets, encoding the intersection data of an arrangement of certain subspaces in abelian Lie groups.

\begin{defn}\label{def:arrangement}
Let $\Lambda=\Z^r$ so that for a connected abelian Lie group $\G$, $\Hom(\Lambda,\G)\cong\G^r$, and an element $\chi\in\Lambda$ determines a group homomorphism $\chi\colon\G^r\to\G$. 

An \textbf{abelian arrangement} is a collection $\A=\A(\G)=\{H_1,\dots,H_n\}$ where, for each $i$, $H_i=\chi_i^{-1}(g_i)$ for some $g_i\in\G$ and primitive $\chi_i\in\Lambda$.

The \textbf{intersection poset} (or \emph{poset of layers}) $\P=\P(\A)$ of an abelian arrangement $\A=\A(\G)$ is the collection of all connected components of intersections $\cap_{i\in T}H_i$ where $T\subseteq[n]$, partially ordered by reverse inclusion. By convention, $\G^r$ is the unique minimal element of $\P(\A)$. 

Given $x\in\P(\A)$, define 
$\Lambda_x:= \Lambda \cap\langle \chi_i \st H_i\supseteq x \rangle_{\R}.$
\end{defn}

\begin{example}\label{ex:arrangement}
Let $\Lambda=\Z^2$ and consider the column vectors of the matrix 
\[\begin{blockarray}{ccc}
\chi_1 & \chi_2 & \chi_3 \\
\begin{block}{[ccc]}
  1 & 0 & 2 \\
  0 & 1 & -1 \\
\end{block}
\end{blockarray}.\]
We will consider various choices of $\G$ and the arrangement $\A(\G)=\{H_1,H_2,H_3\}$ that arises, where $H_i=\chi_i^{-1}(e)$ and $e$ is the identity element of $\G$. Notably, there is an isomorphism of posets $\P(\A(\G))\cong\P(\A(\G\times\R))$, and the following three cases are of most interest.
\begin{enumerate}
\item \label{ex:A:R} When $\G=\R$ or $\G=\C$, $\A(\G)$ is an arrangement of hyperplanes in the vector space $\G^2$. The $i$th column vector is normal to the hyperplane $H_i=\chi_i^{-1}(0)$, and all three hyperplanes meet at the origin. The arrangement is depicted in \Cref{fig:A:R}, and its intersection poset is depicted in \Cref{fig:P:R}. The poset $\P(\A)$ has the structure of a geometric lattice, as is the case for any \emph{linear arrangement} ($\G=\R^b$ for $b>0$) with $0\in H_i$ for each $i$.
\item \label{ex:A:T} When $\G=S^1$ or $\G=\C^\times$, the subspaces $H_i$ are hypertori in the torus $\G^2$, and $\A$ is called a \emph{toric arrangement}. 
Here, the intersection $H_2\cap H_3$ contains points $(t_1,t_2)\in\G^2$ where $t_2=1$ and $t_1^2t_2^{-1}=1$, equivalently $t_1=\pm1$ and $t_2=1$.
This means $H_2\cap H_3=\{u,x\}$, where $u=(1,1)$ and $x=(-1,1)$, corresponding to $H_2\vee H_3 = \{u,x\}$  in $\P(\A)$.
The arrangement is depicted in \Cref{fig:A:T1} and \Cref{fig:A:T2}, and its intersection poset is depicted in \Cref{fig:P:T}.
\item \label{ex:A:E} When $\G=S^1\times S^1$, $\A$ is called an \emph{elliptic arrangement}.
Here, the intersection $H_2\cap H_3$ contains the four points $((\pm1,\pm1),(1,1))\in(S^1\times S^1)^2$.
While we do not depict the arrangement itself, its intersection poset is depicted in \Cref{fig:P:E}.
\end{enumerate}

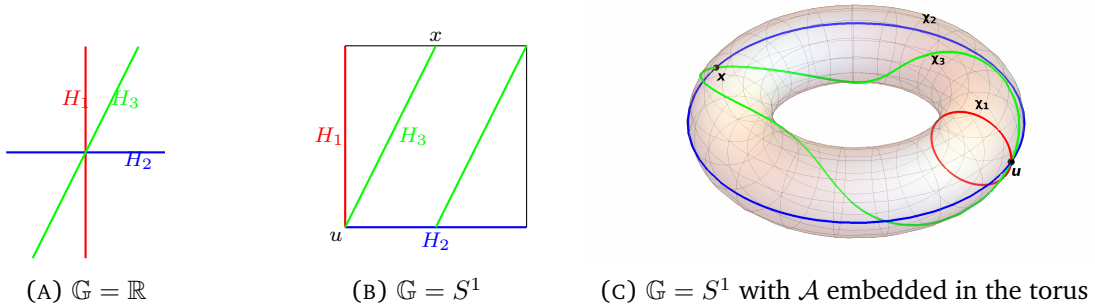
\begin{figure}[ht]
\begin{subfigure}[t]{.25\textwidth}
\centering
\begin{tikzpicture}[scale=.7]
\draw[thick,red,-] (-1,-3)--(-1,1);
\draw[thick,blue,-] (-2.5,-1)--(.5,-1);
\draw[thick,green,-] (-2,-3)--(-1,-1)--(0,1);
\node at (-1.175,0) {\textcolor{red}{\scriptsize$H_1$}};
\node at (0,-1.15) {\textcolor{blue}{\scriptsize$H_2$}};
\node at (-.25,0) {\textcolor{green}{\scriptsize$H_3$}};
\end{tikzpicture}
\caption{$\G=\R$}
\label{fig:A:R}
\end{subfigure}
\begin{subfigure}[t]{.25\textwidth}
\centering
\begin{tikzpicture}[scale=1.2]
\draw[thick,red,-] (-1,-1)--(-1,1);
\draw[-] (-1,1)--(1,1);
\draw[thick,blue,-] (-1,-1)--(1,-1);
\draw[-] (1,-1)--(1,1);
\draw[thick,green,-] (-1,-1)--(0,1);
\draw[thick,green,-] (0,-1)--(1,1);
\node at (-1.1,-1.1) {\scriptsize$u$};
\node at (0,1.1) {\scriptsize$x$};
\node at (-1.175,0) {\textcolor{red}{\scriptsize$H_1$}};
\node at (0,-1.15) {\textcolor{blue}{\scriptsize$H_2$}};
\node at (-.25,0) {\textcolor{green}{\scriptsize$H_3$}};
\end{tikzpicture}
\caption{$\G=S^1$}
\label{fig:A:T1}
\end{subfigure}
\begin{subfigure}[t]{.4\textwidth}
\centering
\includegraphics[trim={0 2cm 0 1cm},clip,width=1\textwidth]{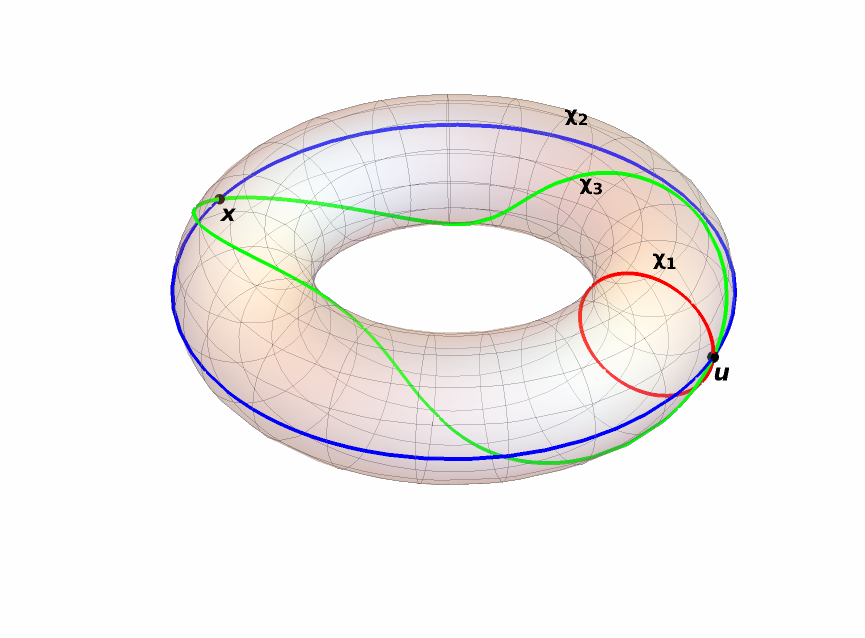} 
\caption{$\G=S^1$ with $\A$ embedded in the torus}
\label{fig:A:T2}
\end{subfigure}
\caption{Pictures of a hyperlane arrangement and a toric arrangement.}
\end{figure}

\begin{figure}[ht]
\begin{subfigure}[t]{.3\textwidth}
\centering
\begin{tikzpicture}
\node (0) at (0,0) {\scriptsize $\G^2$};
\node (1) at (-1,1.5) {\scriptsize $H_1$};
\node (2) at (0,1.5) {\scriptsize $H_2$};
\node (3) at (1,1.5) {\scriptsize $H_3$};
\node (u) at (0,3) {0}; 
\foreach \x in {1,2,3} {
\draw[-] (0.north)--(\x.south);
\draw[-] (u.south)--(\x.north);
};
\end{tikzpicture}
\caption{$\G=\R^b$}
\label{fig:P:R}
\end{subfigure}
\begin{subfigure}[t]{.3\textwidth}
\centering
\begin{tikzpicture}
\node (0) at (0,0) {\scriptsize $\G^2$};
\node (1) at (-1,1.5) {\scriptsize $H_1$};
\node (2) at (0,1.5) {\scriptsize $H_2$};
\node (3) at (1,1.5) {\scriptsize $H_3$};
\node (u) at (-1,3) {\scriptsize $u$};
\node (x) at (1,3) {\scriptsize $x$};
\foreach \x in {1,2,3} {
\draw[-] (0.north)--(\x.south);
\draw[-] (u.south)--(\x.north);
};
\foreach \x in {2,3} {
\draw[-] (x.south)--(\x.north);
};
\end{tikzpicture}
\caption{$\G=S^1\times \R^b$}
\label{fig:P:T}
\end{subfigure}
\begin{subfigure}[t]{.3\textwidth}
\centering
\begin{tikzpicture}
\node (0) at (0,0) {\scriptsize $\G^2$};
\node (1) at (-1,1.5) {\scriptsize $H_1$};
\node (2) at (0,1.5) {\scriptsize $H_2$};
\node (3) at (1,1.5) {\scriptsize $H_3$};
\node (u) at (-1,3) {\scriptsize $p_1$};
\node (y) at (0,3) {\scriptsize $p_2$};
\node (x) at (1,3) {\scriptsize $p_3$};
\node (z) at (2,3) {\scriptsize $p_4$};
\foreach \x in {1,2,3} {
\draw[-] (0.north)--(\x.south);
\draw[-] (u.south)--(\x.north);
};
\foreach \x in {2,3} {
\draw[-] (x.south)--(\x.north);
\draw[-] (y.south)--(\x.north);
\draw[-] (z.south)--(\x.north);
};
\end{tikzpicture}
\caption{$\G=S^1\times S^1\times \R^b$}
\label{fig:P:E}
\end{subfigure}
\caption{The intersection posets of the arrangements in \Cref{ex:arrangement}}
\end{figure}
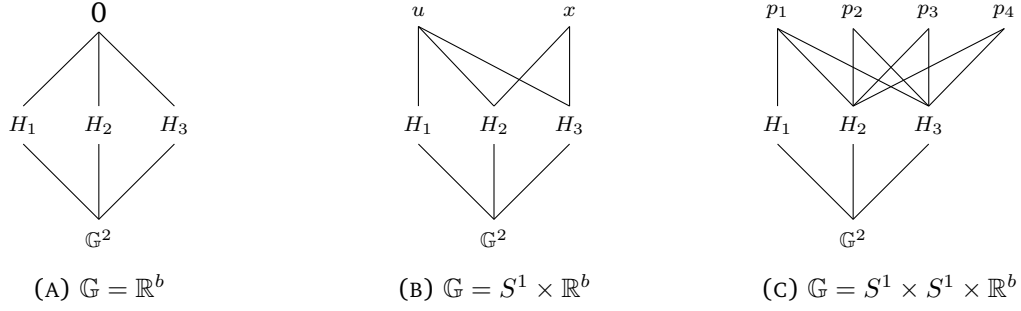

In each of these examples, $\Lambda_x=\langle \chi_i\rangle_\Z$ when $x=H_i$, and $\Lambda_x=\Lambda$ when $x$ is one of the intersection points.
Note that $\Lambda_x$ is not in general equal to the integer lattice $\langle \chi_i \st H_i\supseteq x\rangle_\Z$.
For instance, taking $x$ as the element labeled so in \Cref{fig:P:T}, one has $\langle \chi_2,\chi_3\rangle_\Z$ as an index-2 subgroup of $\Lambda_x=\Lambda$.
\end{example}

\begin{remark}
Let us make two important remarks on the integer lattice $\Lambda_x$ for an element $x\in\P(\A)$.
First, if $x$ is a connected component of $\cap_{i\in T}H_i$, then $\Lambda_x=\Lambda\cap\langle \chi_i\st i\in T\rangle_{\R}$, meaning that $\Lambda_x$ does depend on the choice of $T$ for which $x$ is a component of the intersection $\cap_{i\in T}H_i$.

For the second remark, by translating when necessary, we may assume that $x$ is a subgroup of $\G^r$. Then any $\lambda\in\Lambda_x$ has $x\subseteq\ker(\lambda)$, and choosing a $\Z$-basis $(\lambda_1,\dots,\lambda_k)$ for $\Lambda_x$ allows us to realize $x$ as the kernel of a group homomorphism $\G^r\to\G^k$ with $g\mapsto(\lambda_1(g),\dots,\lambda_k(g))$. 
\end{remark}

Let $\A=\A(\G)$ be an abelian arrangement in $\G^r$, where $\G$ is a connected abelian Lie group of dimension $d$. 
We will be interested in the topology of the arrangement complement
\[\M(\A) := \G^r\setminus \cup_{H\in\A} H.\]
This fits into the context of this paper because the intersection poset $\P(\A)$ is a geometric poset, as we recall in \Cref{thm:P(A)geom}, and there is intuitive topological understanding for why $\P(\A)$ is locally geometric, explained as follows.
Let $X\in\P(\A)$. We say that $g\in X$ is a \emph{generic point} of $X$ if $g\notin Y$ for any $Y\in\P_{>X}$. 
For such $g$, we define a linear arrangement $\A_X^{(g)}$ in the tangent space $T_g\G^r\cong(\R^d)^r$ as the collection
\[ \A_X^{(g)} = \{T_gH \st H\in\A, H\ni g\}. \]
The arrangement $\A_X^{(g)}$ is called the \emph{localization} of $\A$ at $g$ because there is a small neighborhood $U\subseteq\G^r$ of $g$ for which $U\cap \M(\A)\cong \M(\A_X^{(g)})$. 
Up to canonical homeomorphism (by translation), the topological space $\M(\A_X^{(g)})$ is independent of the choice of generic $g\in X$. 
We will thus abbreviate
\begin{equation}\label{eq:localization} 
\M_X := \M(\A_X^{(g)}) \qquad \text{ for some (any) generic point } g\in X 
\end{equation}
Since the localization $\A_X^{(g)}$ is a linear arrangement, its intersection poset is a geometric lattice and, since  $\P(\A_X^{(g)})\cong \P(\A)_{\leq X}$, this implies $\P(\A)$ is locally geometric. 
For example, the linear arrangement in \Cref{ex:arrangement}\eqref{ex:A:R} is the localization of the toric arrangement in \Cref{ex:arrangement}\eqref{ex:A:T} at the point $u$, and the interval under $u$ in \Cref{fig:P:T} is isomorphic to the poset in \Cref{fig:P:R}.

\begin{theorem}[{\cite[Cor 4.4.7]{BD}}]\label{thm:P(A)geom}
Let $\A=\A(\G)$ be an abelian arrangement in $\G^r$, where $\G$ is a connected abelian Lie group of dimension $d$.
The intersection poset $\P(\A)$ is a geometric poset with $\rk(x)=(dr-\dim(x))/d$ for all $x\in\P(\A)$.
\end{theorem}

\subsection{A topological Orlik--Solomon algebra}\label{sec:topOS}
When $\G=\R^d$, the cohomology ring of an arrangement complement $\M(\A)$ depends only on the parity of $d$ and the combinatorics of the arrangement. When $d=2l$ is even (so that $\G\cong\C^l$), the cohomology ring is isomorphic to the Orlik--Solomon algebra of the intersection poset $\P(\A)$ with generators in degree $d-1$ (\cite[Thm 5.2]{OS} when $l=1$ and \cite[Cor 5.6]{dLS} when $l>1$), while the $d$ odd case behaves differently. We thus focus here on the even-dimensional case only, equivalently when $\G$ is a complex abelian Lie group.
Although the cohomology ring of an arrangement complement is isomorphic to an Orlik--Solomon algebra in the case $\G\cong\C^l$, more generally the cohomology of the ambient space $\G^r$ affects the cohomology of the arrangement complement. 

This leads us to first define a topological Orlik--Solomon algebra, which combines both the combinatorial and topological inputs of an arrangement in $\G^r$.
This generalizes the ``graded Orlik--Solomon algebra'' for toric arrangements  in \cite{pagaria} and the model for elliptic arrangements  in \cite{bibby-elliptic}. 
There is also a ``generalized Orlik--Solomon algebra'' produced by other means to compute cohomology of more general manifold arrangements in \cite{CLW}; 
a key advantage of ours is the explicit presentation.

\begin{defn}[Topological OS Algebra]\label{def:B}
Let $\G$ be a connected complex abelian Lie group with real dimension $d=\dim\G$, $\Lambda=\Z^r$, and let $\P=\P(\mathcal{A})$ be the intersection poset for an arrangement $\mathcal{A}$ in $\G^r$. 
Adjust the grading on the Orlik--Solomon algebra $A^*(\P)$ so that the degree of the generator $e_\alpha$ is equal to $|\alpha|\cdot(d-1)$. 
Define a bigraded algebra $B=B(\mathcal{A})$ as the quotient of $H^*(\G^r;\Z)\otimes A^*(\P)$ by the homogeneous ideal 
\[ \langle
\chi^*(z)\otimes e_{(x,T)} \text{ where } (x,T)\in\I(\P), z\in H^{1}(\G;\Z), \chi\in{\Lambda}_x
\rangle.\]
\end{defn}

Supposing $\G\cong (S^1)^a\times\R^b$, we can write an explicit list of generators and relations for $B(\A)$.
First let $z_1,\dots,z_a\in H^1(\G;\Z)$ generate $H^*(\G;\Z)$ as an exterior algebra, and let $\pi_i:\G^r\to\G$ be the projection onto the $i$th coordinate so that the elements $z_j[i]:=\pi_i^*(z_j)$ generate $H^*(\G^r;\Z)$ as an exterior algebra.
Let $F$ be the free graded-commutative algebra generated by 
\begin{itemize}
\item $z_j[i]$ in bidegree $(1,0)$, where $1\leq j\leq a$ and $1\leq i\leq r$, and 
\item $e_\alpha$ in bidegree $(0,|\alpha|(d-1))$, where $\alpha\in\I\setminus\{\hat{0}\}$.
\end{itemize}
Then the algebra $B(\A)$ is the quotient of $F$ by the ideal generated by
    \begin{enumerate}
    \item[\eqref{OSi}] $e_\alpha e_\beta$ when $\alpha\meet\beta\neq\{\hat{0}\}$;
    \item[\eqref{OSii}] $e_\alpha e_\beta-\sgn(\alpha,\beta)\sum\limits_{\gamma\in(\alpha\join\beta)\cap\I}e_\gamma$ when $\alpha\meet\beta=\{\hat{0}\}$; 
    \item[\eqref{OSiii}] $\sum\limits_{a\in\at(\zeta)}\sgn(a,\omega\setminus a)e_{\omega\setminus a}$ when $\omega$ is a $\zeta$-unicircular element; and
\item[\mylabel{Biv}{iv}] $\sum_{i=1}^r \chi[i]z_j[i]e_{(x,T)}$ where $1\leq j\leq a$, $(x,T)\in\I(\P)$, and $\chi=(\chi[1],\dots,\chi[r])\in{\Lambda}_x$.
\end{enumerate}

\begin{example}\label{ex:ellipticB}
Let $\G=S^1\times S^1$ and consider the arrangement $\A=\{H_1,H_2,H_3\}$ from \Cref{ex:arrangement}\eqref{ex:A:E}. 
The algebra $B(\A)$ is generated by 
$z_1[1], z_2[1], z_1[2], z_2[2]$ in degree $(1,0)$; $e_1,e_2,e_3$ in degree $(0,1)$; and 
$e_{(p_1,12)}, e_{(p_1,13)}, e_{(p_1,23)}, e_{(p_2,23)}, e_{(p_3,23)}, e_{(p_4,23)}$ in degree $(0,2)$.
The relations here are given by $e^2_\alpha$ for all $\alpha$;
$e_\alpha e_\beta$ and $z_j[i]e_\beta$ when $\rk(\beta)=2$; and
\[ e_1e_2-e_{(p_1,12)},\quad  e_1e_3-e_{(p_1,13)},\quad e_2e_3-\sum_{i=1}^4 e_{(p_i,23)}, 
\quad  e_{(p_1,23)}-e_{(p_1,13)}+e_{(p_1,12)} \]
\[ z_1[1]e_1,\quad z_2[1]e_1,\quad z_1[2]e_2,\quad z_2[2]e_2,\quad (2z_1[1]-z_1[2])e_3,\quad (2z_2[1]-z_2[2])e_3\]
\end{example}

\begin{remark}\label{rmk:geom-joinindep}
Since the poset $\P(\A)$ is in fact geometric, \cite[Prop 4.5(2)]{bibby} implies that, for all $\alpha,\beta\in S$, either $(\alpha\vee \beta)\cap\I=\varnothing$ or $\alpha\vee \beta\subseteq\I$. This means relations \eqref{OSi} and \eqref{OSii} can be rewritten as 
\begin{enumerate}
\item[(I)] $e_\alpha e_\beta$ when $\alpha\meet\beta\neq\{\hat{0}\}$ or $\alpha\vee\beta\not\subseteq\I$; and
\item[(II)] $e_\alpha e_\beta-\sgn(\alpha,\beta)\sum\limits_{\gamma\in\alpha\join\beta}e_\gamma$ when $\alpha\meet\beta=\{\hat{0}\}$ and $\alpha\vee \beta\subseteq\I$.
\end{enumerate}
\end{remark}

The following lemma will provide us with a useful perspective on the defining ideal, and in turn a vector space decomposition of $B(\A)$ building on the decomposition of  $A(\P)$ in \Cref{cor:brieskorn}.

\begin{lemma}\label{prop:K_x}
Let $\G$ be a connected complex abelian Lie group and $\P=\P(\mathcal{A})$ the intersection poset for an arrangement $\mathcal{A}$ in $\G^r$. For an element $x\in\P$ let $K_x=\ker(H^*(\G^r;\Z)\to H^*(x;\Z))$. Then: 
\begin{enumerate}
\item \label{prop:K_x:1} for any $x\in\P$, the ideal $K_x\subseteq H^*(\G^r;\Z)$ is generated by $\chi^*(z)$ where $z\in H^{1}(\G;\Z)$ and $\chi\in{\Lambda}_x$,
\item\label{prop:K_x:2} for any $x\in\P$, there is a ring isomorphism $H^*(\G^r;\Z)/K_x\cong H^*(x;\Z)$, and
\item \label{prop:K_x:3} if $x,y\in\P$ with $x\leq y$, then $K_x\subseteq K_y$.
\end{enumerate}
\end{lemma}
\begin{proof}
As $x$ is a translate of some connected Lie subgroup of $\G^r$, and translation in Lie groups is homotopic to the identity, we may assume that $x$ is a connected Lie subgroup of $\G^r$. 
Let $\{\lambda_1,\dots,\lambda_k\}$ be a basis for the  integer lattice ${\Lambda}_x$ and consider the surjective homomorphism
\[\lambda_x:=(\lambda_1,\dots,\lambda_k):\G^r\to \G^k,\]
with kernel equal to $x$.
The fiber bundle $x\to \G^r\stackrel{\lambda_x}{\longrightarrow} \G^k$ induces
a split short exact sequence on fundamental groups 
$0\to\pi_1(x)\to\pi_1(\G^r)\to\pi_1(\G^k)\to0.$
As these fundamental groups are free abelian and isomorphic to the respective first integral homology groups, applying $\operatorname{Hom}(-,\Z)$ yields a short exact sequences on the first cohomology groups
\begin{equation*}
0\to H^1(\G^k;\Z)\stackrel{\lambda_x^*}{\longrightarrow} H^1(\G^r;\Z)\to H^1(x;\Z)\to0.
\end{equation*}
Denote $K_x^1 = \ker(H^1(\G^r;\Z)\to H^1(x;\Z)) = \operatorname{im}(\lambda_x^*:H^1(\G^k;\Z)\to H^1(\G^r;\Z))$. 
As the cohomology ring of an abelian Lie group is canonically isomorphic to the exterior algebra generated by the first cohomology group, it follows 
that $K_x$ is generated as an ideal by $K_x^1$.
This means that $K_x$ is generated by $\lambda_x^*(w)$ where $w\in H^1(\G^k;\Z)$.
Using the K\"unneth Formula and the projections $\pi_i\colon \G^k\to \G$ onto each coordinate, 
this implies that $K_x$ is generated by 
\[\lambda_x^*(\pi_i^*(z)) = (\pi_i\circ \lambda_x)^*(z) = \lambda_i^*(z),\]
where $z\in H^1(\G;\Z)$ and $1\leq i \leq k$.
Since $\lambda_1,\dots,\lambda_k$ is a basis for ${\Lambda}_x$, we conclude that $K_x$ is generated by $\chi^*(z)$ where $z\in H^1(\G;\Z)$ and $\chi\in{\Lambda}_x$ as desired for the first claim.

The second claim follows from the first by the First Isomorphism Theorem, noting that 
surjectivity of $H^1(\G^r;\Z)\to H^1(x;\Z)$ above implies surjectivity of $H^*(\G^r;\Z)\to H^*(x;\Z)\cong\bigwedge H^1(x;\Z)$.

Lastly, the third claim follows from the first with the observation that if $x\leq y$, then ${\Lambda}_x\subseteq{\Lambda}_y$.
\end{proof}

Building on the analogous theorem for $A(\P)$ in \Cref{cor:brieskorn}, we establish a vector decomposition and Hilbert series formula for $B(\A)$.

\begin{prop}[$\P$-decomposition of $B(\A)$]\label{lem:Bvs}
Let $\G=(S^1)^a\times \R^b$ with $a+b$ even, and let $\A$ an  arrangement in $\G^r$ with intersection poset $\P=\P(\A)$.
Then $B(\A)$ is a free $\Z$-module with decomposition
\[ B(\mathcal{A}) \cong \bigoplus_{x\in\P} H^*(x;\Z)\otimes A^{\rk(x)}(\P_{\leq x}) \]
where the degree-$(p,j(d-1))$ component is \[B^{p,j(d-1)}(\A)\cong \bigoplus_{\substack{x\in\P \\ \rk(x)=j}} H^p(x;\Z)\otimes A^{j}(\P_{\leq x})\]
and $B^{p,q}=0$ when $q$ is not divisible by $d-1$.
The bigraded Hilbert series of $B(\A)$ is
\[\sum_{p,q\geq0} \dim(B^{p,q})t^pu^q = (1+t)^{a(r-\rk(\P))}(-u)^{\rk(\P)(a+b-1)} \chi_{\P}\left( -\dfrac{(1+t)^a}{u^{a+b-1}} \right) \]
where $\chi_\P$ is the characteristic polynomial of $\P$.
\end{prop}
\begin{proof}
The $\nbc$ basis for $A(\P)$ from \Cref{cor:vsbasis} provides a decomposition $H^*(\G^r;\Z)\otimes A(\P)\cong\bigoplus_{\beta\in\nbc} H^*(\G^r;\Z)\otimes\Z e_{\beta}$. 
Let $J\subseteq H^*(\G^r;\Z)\otimes A(\P)$ be the defining ideal of $B(\A)$, which by \Cref{prop:K_x}\eqref{prop:K_x:1} can be written as
$J=\sum_{(x,T)\in\I} K_x\otimes\langle e_{(x,T)}\rangle$. 

Consider $(x,T)\in\I$. Any element of the ideal $\langle e_{(x,T)}\rangle\subseteq A(\P)$ can be expressed as a $\Z$-linear combination of $e_{(y,S)}$ where $(y,S)\in\nbc$ and $y\geq x$. Since $K_x\subseteq K_y$ whenever $y\geq x$ by \Cref{prop:K_x}\eqref{prop:K_x:3}, we conclude that $K_x\otimes\langle e_{(x,T)}\rangle \subseteq \sum_{(y,S)\in\nbc} K_y\otimes \Z e_{(y,S)}$.
It follows that
\begin{equation}\label{eq:idealJ} 
J = \sum_{(x,T)\in\I} K_x\otimes\langle e_{(x,T)}\rangle = \sum_{(x,T)\in\nbc} K_x\otimes\Z  e_{(x,T)}
\end{equation}
Then applying \eqref{eq:idealJ}, \Cref{prop:K_x}\eqref{prop:K_x:2}, and \Cref{cor:brieskorn}, in order, we obtain
\[
B(\A) = 
\dfrac{\bigoplus_{(x,T)\in\nbc} H^*(\G^r;\Z)\otimes\Z e_{(x,T)}}{\sum_{(x,T)\in\nbc} K_x\otimes \Z e_{(x,T)}} 
\cong \bigoplus_{(x,T)\in\nbc} H^*(x;\Z)\otimes \Z e_{(x,T)} 
\cong \bigoplus_{x\in\P} H^*(x;\Z)\otimes A^{\rk(x)}(\P_{\leq x}) 
\]

The Hilbert series of $B(\A)$ is thus
\begin{align*}
\sum_{p,q\geq 0} \dim B^{p,q}(\A) t^pu^q
&= \sum_{x\in\P} \sum_{p\geq 0} \dim(H^p(x;\Z))  t^p \dim(A^{\rk(x)}(\P_{\leq x})) u^{\rk(x)(a+b-1)} \\
&= \sum_{x\in\P} (1+t)^{a(r-\rk(x))}(-1)^{\rk(x)}\mu(x)u^{\rk(x)(a+b-1)} \\
&= (1+t)^{a(r-\rk(\P))}(-u)^{\rk(\P)(a+b-1)} \sum_{x\in\P} \mu(x) \left( -\dfrac{(1+t)^a}{u^{a+b-1}} \right)^{\rk(\P)-\rk(x)} \\
&= (1+t)^{a(r-\rk(\P))}(-u)^{\rk(\P)(a+b-1)} \chi_{\P}\left( -\dfrac{(1+t)^a}{u^{a+b-1}} \right) 
\end{align*}
\end{proof}

\begin{example}
The characteristic polynomial of the poset $\P(\A)$ in \Cref{fig:P:E} is $\chi_\P(t)=t^2-3t+5$, and the associated algebra $B=B(\A)$ from \Cref{ex:ellipticB} has bigraded Hilbert series 
\[ \sum_{p,q\geq0} \dim B^{p,q} t^pu^q = 
(-u)^2\chi_\P\left(-\frac{(1+t)^2}{u}\right)
=1+4t+6t^2+4t^3+t^4+3u+6tu+3t^2u+5u^2.\]
\end{example}

\subsection{The Leray spectral sequence}\label{sec:Leray}
When $\G$ is complex and $\A$ is an arrangement in $\G^r$, we can relate the topological Orlik--Solomon algebra $B(\A)$ to the cohomology of the arrangement complement $\M(\A)$. 
The main tool is the Leray spectral sequence for the inclusion 
$f\colon \M(\A)\to \G^r$ given by 
\[ E_2^{p,q} = H^p(\G^r;R^qf_*\Z_{\M(\A)}) \Longrightarrow H^{p+q}(\M(\A);\Z), \]
where the higher direct image sheaf $R^qf_*\Z_{\M(\A)}$ is the sheafification of the presheaf on $\G^r$ given by $U\mapsto H^q(U\cap \M(\A);\Z)$.
We adopt the same method of \cite{totaro} for configuration spaces and \cite{bibby-elliptic,pagaria} for $\dim\G=2$
to prove the following.

\begin{theorem}[OS algebra in the Leray spectral sequence]\label{thm:Leray}
Let $\G$ be a connected complex abelian Lie group, and let $\mathcal{A}$ be an arrangement in $\G^r$ with complement $\M(\mathcal{A})$.
The bigraded algebra $B(\mathcal{A})$ is isomorphic to the second page of the Leray spectral sequence for the inclusion $f\colon \M(\mathcal{A})\into \G^r$.
\end{theorem}
\begin{proof}
Denote by $E_2(\mathcal{A})$ the bigraded algebra on the second page of the Leray spectral sequence.
We start by examining the sheaf $R^qf_*\Z_{\M(\A)}$.
Since, for small neighborhoods $U$, $U\cap \M(\A)$ is homeomorphic to the complement of a linear arrangement, the group $H^q(U\cap \M(\A);\Z)$ is trivial unless $d-1$ divides $q$. In particular, $R^qf_*\Z_{\M(\A)}=0$ unless $d-1$ divides $q$, so let us assume $q=j(d-1)$. The goal is to provide a useful description of the sheaf $R^{j(d-1)}f_*\Z_{\M(\A)}$ in \eqref{eq:sheafiso} and then of $E_2(\A)$ in \eqref{eq:Edecomp}, after which we describe the algebra structure and finally compare to $B(\A)$.

For $x\in\P(\A)$ with $\rk(x)=j$, let $\A|_x=\{H\in\A\st x\subseteq H\}$, an arrangement in $\G^r$ with complement $\M(\A|_x)$. Let $i_x\colon x\into\G^r$ and $f_x\colon \M(\A|_x)\into\G^r$ be the inclusions.
The sheaf $(i_x)^*R^{j(d-1)}(f_x)_*\Z_{\M(\A|_x)}$ is a locally constant sheaf on $x$, because every point of $x$ has a neighborhood $U$ for which sections are $H^{j(d-1)}(U\cap \M(\A|_x);\Z)\cong H^{j(d-1)}(\M_x;\Z)\cong A^j(\P_{\leq x})$, where $\M_x$ is the localization from \eqref{eq:localization}. It is in fact a constant sheaf since stalks can be canonically identified via translation.
This yields an isomorphism 
\begin{equation}\label{eq:sheaf1}
A^j(\P_{\leq x})\otimes(i_x)_*\Z_x
\cong
(i_x)_*(i_x)^*R^{j(d-1)}(f_x)_*\Z_{\M(\A|_x)}.
\end{equation}
Next, the inclusions $i_y\colon y\into\G^r$ induce a natural morphism of sheaves
\[R^{j(d-1)}(f_x)_*\Z_{\M(\A|_x)}\to \bigoplus_{\substack{y\in\P(\A|_x) \\ \rk(y)=j}} (i_y)_*(i_y)^*R^{j(d-1)}(f_x)_*\Z_{\M(\A|_x)}. \]
On either side, the stalk at a point is trivial unless it is a point in some $y\in\P(\A|_x)$ with $\rk(y)=j$, in which case the stalk is $H^{j(d-1)}(\M_y)$. 
Being an isomorphism on stalks, it is necessarily an isomorphism of sheaves, and restricting its inverse to the summand indexed by $x$ yields a morphism of sheaves
\begin{equation}\label{eq:sheaf2}
(i_x)_*(i_x)^*R^{j(d-1)}(f_x)_*\Z_{\M(\A|_x)}
\to
R^{j(d-1)}(f_x)_*\Z_{\M(\A|_x)}.
\end{equation}
The inclusion $\M(\A)\into \M(\A|_x)$ induces a morphism $R^{j(d-1)}(f_x)_*\Z_{\M(\A|_x)}\to R^{j(d-1)}f_*\Z_{\M(\A)}$, and precomposing with \eqref{eq:sheaf1} and \eqref{eq:sheaf2} yields a morphism
\begin{equation}\label{eq:sheaf3} A^j(\P_{\leq x})\otimes(i_x)_*\Z_x
\to R^{j(d-1)}f_*\Z_{\M(\A)}. \end{equation}
Finally, the sum of morphisms in \eqref{eq:sheaf3} over all $x\in\P(\A)$ with $\rk(x)=j$ yields a morphism
\begin{equation}\label{eq:sheafiso}
\bigoplus_{\substack{x\in\P(\A) \\ \rk(x)=j}}
A^j(\P_{\leq x})\otimes(i_x)_*\Z_x
\to R^{j(d-1)}f_*\Z_{\M(\A)}.
\end{equation}
The morphism \eqref{eq:sheafiso} is a sheaf isomorphism because it is an isomorphism on stalks, noting that the stalk at a generic point of $y\in\P(\A)$ is 
\[ H^{j(d-1)}(\M_y;\Z) \cong A^j(\P_{\leq y}) 
\cong\bigoplus_{\substack{x\in{\P_{\leq y}} \\ \rk(x)=j}} A^j(\P_{\leq x}) \]
by the $\P$-decomposition.
The sheaf isomorphism \eqref{eq:sheafiso} induces a vector space decomposition
\begin{equation}\label{eq:Edecomp}
E_2^{p,j(d-1)} = H^p(\G^r; R^{j(d-1)}f_*\Z_{\M(\A)})
\cong \bigoplus_{\substack{x\in\P(\A) \\ \rk(x)=j}} H^p(x;\Z)\otimes A^j(\P_{\leq x}).
\end{equation}

The algebra structure in $E_2(\mathcal{A})$ is described as follows. 
For $x_1,x_2\in\P$ and $z\in x_1\vee x_2$, we have inclusions of algebras $j_k\colon A(\P_{\leq x_k})\into A(\P_{\leq z})$ and inclusions of topological spaces $i_k\colon z\into x_k$. 
Then, given elements
$h_k\otimes a_k \in H^{p_k}(x_k;\Q)\otimes A^{\rk(x_k)}(\P_{\leq x_k})$
for $k=1,2$, 
their product in $E_2(\mathcal{A})$ is 
\[ (h_1\otimes a_1)\cdot (h_2\otimes a_2)
= (-1)^{p_2\rk(x_1)} \sum_{z\in x_1\vee x_2} \left( (i_1^*(h_1)\smile i_2^*(h_2)) \otimes (j_1(a_1)\cdot j_2(a_2)) \right), \]
where the products on the right-hand side are the cup product $\smile$ in $H^*(z;\Z)$ and the multiplication in $A(\P_{\leq z})$.
In particular, the subalgebras $E_2^{0,*}\cong A(\P)$ and $E_2^{*,0}\cong H^*(\G^r;\Z)$ together generate $E_2(\A)$ as an algebra, providing a surjective algebra homomorphism $\phi\colon H^*(\G^r;\Z)\otimes A(\P)\to E_2(\A)$. 
Now when $(x,T)\in\I$, $\chi\in{\Lambda}_x$, and $z\in H^1(\G;\Z)$, we have 
$\phi(\chi^*(z)\otimes e_{(x,T)})
= (i_x)^*(\chi^*(z)) \otimes e_{(x,T)}
= 0$
by \Cref{prop:K_x}. 
This means that $\phi$ induces a surjective algebra homomorphism on the quotient $B(\A)\to E_2(\A)$, which is necessarily an isomorphism since it is an isomorphism on bigraded vector spaces via the decompositions \eqref{eq:Edecomp} and \Cref{lem:Bvs}. 
\end{proof}

\subsection{Cohomology in the noncompact case}\label{sec:noncpt}
The main goal of this section is \Cref{thm:noncpt} relating our topological Orlik--Solomon algebra to the cohomology ring of the arrangement complement, when $\G$ is noncompact. 
We shall see that the Hilbert series of the two algebras agree, but the algebra structure is more subtle. 
The cohomology ring presentation which we compare to is given in \cite[Theorem 5.9]{BPP} and falls into two cases.
We will show in \Cref{thm:noncpt} that, in the case which excludes $\G\cong\G'\times\C^\times$ for some compact $\G'$, the cohomology ring is isomorphic to our topological Orlik--Solomon algebra $B(\A)$. Let us start by recalling the ring presentation in this case for reference.

\begin{theorem}\label{thm:BPP}\cite[Thm 5.9]{BPP}
Let $\G$ be a noncompact connected complex abelian Lie group where $\G\not\cong \G'\times\C^\times$ with $\G'$ compact (possibly trivial). Let $\A$ be an arrangement in $\G^r$ with poset of layers $\P(\A)$ and associated matroid scheme $(S,\I)$.
Define the free $H^*(\G^r;\Z)$-module $R(\A)$ generated by $w_\alpha$ for $\alpha\in\I$, and endow $R(\A)$ with a ring structure by defining the multiplication
\begin{equation}\label{eqn:multi}\tag{I-II}
w_\alpha w_\beta=
\begin{cases}
\sgn(\alpha,\beta)\sum\limits_{\gamma\in\alpha\join\beta}w_\gamma & \text{when $\alpha\join\beta\subseteq\I$ and $\alpha\meet\beta=\{\hat{0}\}$,}\\
0 & \text{otherwise.}
\end{cases}
\end{equation}
The integer cohomology of the complement $H^*(\M(\A);\Z)$ is the quotient of $R(\A)$ by the following:
\begin{enumerate}
\item[\mylabel{eq:BPP2}{III}] $\displaystyle\sum_{a\in\at(\zeta)}\sgn(a,\omega\setminus a)w_{\omega\setminus a}$, when $\omega$ is $\zeta$-unicircular. 
\item[\mylabel{eq:BPP1}{IV}] $z\cdot w_{(x,T)}=0$, when $(x,T)\in\I$ and $z\in K_x:=\ker(H^*(G^r;\Z)\to H^*(x;\Z))$.
\end{enumerate}
\end{theorem}
Several remarks about the above cited theorem are in order.
The case $\G\cong\G'\times\C^\times$ with $\G'$ compact is omitted here because the analogous relation \eqref{eq:BPP2} becomes more complicated, making the ring structure appear quite different from our Orlik--Solomon algebra.
Lastly, we note that the sign $\sgn(a,\omega\setminus a)$ in \eqref{eq:BPP2} is a correction of the sign $\sgn(a,\zeta\setminus a)$ written in \cite[Thm 5.9]{BPP}, and we use \Cref{prop:rank-uni} to write this relation in terms of unicircular elements.

\begin{theorem}[Cohomology in noncompact case]\label{thm:noncpt}
Let $\G$ be a noncompact connected complex abelian Lie group and $\mathcal{A}$ an arrangement in $\G^r$ with complement $\M(\mathcal{A})$. 
There is a filtration $F$ on the integer cohomology of $\M(\mathcal{A})$ whose associated graded algebra $\mathrm{gr}_FH^*(\M(\mathcal{A});\Z)$ is isomorphic to $B(\mathcal{A})$ as a graded algebra.
Except when $\G\cong \G'\times\C^\times$ with $\G'$ compact, this filtration is trivial and the integer cohomology of $\M(\mathcal{A})$ is isomorphic as a graded algebra to $B(\mathcal{A})$.
\end{theorem}
\begin{proof}
We know from \Cref{thm:Leray} that $B(\A)$ is isomorphic to the second page of the Leray spectral sequence. 
Additionally, when $\G$ is noncompact, both $B(\A)$ and $H^*(M(\A);\Z)$ are free $\Z$-modules with matching Hilbert series (via \Cref{lem:Bvs} and \cite[Thm 7.7]{LTY}).
Together, this implies that $B(\A)\cong E_2=E_\infty=\mathrm{gr}_FH^*(\M(\A);\Z)$ where $F$ is the Leray filtration.

To prove the second claim, we exhibit an algebra isomorphism using the cohomology presentation in \Cref{thm:BPP}. 
We start with the surjective algebra map $\phi_E\colon H^*(\G^r;\Z)\otimes E(\P)\to H^*(\M(\A);\Z)$ given by $1\otimes e_\alpha\mapsto w_\alpha$.
Using \Cref{rmk:geom-joinindep} to compare relations \eqref{eqn:multi} and \eqref{eq:BPP2} in $H^*(\M(\A);\Z)$ with the defining ideal of $A(\P)$ in \Cref{defn:OS}, the map $\phi_E$ induces a surjective algebra homomorphism $\phi_A\colon H^*(\G^r;\Z)\otimes A(\P)\to H^*(\M(\A);\Z)$. Using \Cref{prop:K_x} to compare the defining ideal of $B(\A)$ in \Cref{def:B} with \eqref{eq:BPP1}, we conclude that $\phi_A$ induces a surjective algebra homomorphism $\phi_B\colon B(\A)\to H^*(\M(\A);\Z)$. 
Since the Hilbert series of $B(\A)$ matches that of $H^*(\M(\A);\Z)$, the map $\phi_B$ is necessarily an isomorphism of algebras.
\end{proof}

\subsection{Cohomology in the compact case}\label{sec:cpt}
Less is known about the cohomology ring in the compact case, and we show here how our topological Orlik--Solomon algebra is a tool for computing the cohomology. We equip $B(\A)$ with a differential $\delta$, making it a differential graded algebra. We then show that, with $\Q$-coefficients, this 
is the only nontrivial differential in the Leray spectral sequence and
\begin{equation}\label{eq:H*iso}
H^*(\M(\A);\Q)\cong H^*(B(\A)\otimes\Q,\delta).
\end{equation}
The proof uses the fact that the Leray spectral sequence is compatible with the mixed Hodge structure. In fact, this Hodge theory perspective allows us to conclude that the isomorphism in \eqref{eq:H*iso} is an isomorphism of \emph{algebras}. 
While our methods follow that of Totaro \cite{totaro} for configuration spaces and \cite{bibby-elliptic} for $\G$ a complex elliptic curve, both of those papers only encountered geometric lattices in the combinatorics.

\begin{theorem}[OS model for cohomology]\label{thm:cpt}
Let $\G$ be a compact connected complex abelian Lie group with real dimension $d$. Let $\A$ be an arrangement in $\G^r$ defined by primitive elements $\chi_1,\dots,\chi_n\in\Z^r$, and write $\chi_k=(\chi_k[1],\dots,\chi_k[r])$.
The rational cohomology of the arrangement complement $\M(\mathcal{A})$ is isomorphic as a graded algebra to the cohomology of $B(\mathcal{A})\otimes\Q$ with respect to the differential $\delta$ of bidegree $(d,1-d)$ determined by $\delta(z_j[i])=0$ for $1\leq j\leq d$, $1\leq i\leq r$, and for $(x,T)$ independent
\begin{equation}\label{eq:delta}
\delta(e_{(x,T)}) = \sum_{k\in T} \frac{\sgn(k,T\setminus k)}{c(T,k)}\left(\sum_{i=1}^r \chi_k[i]z_1[i]\right)\cdots\left(\sum_{i=1}^r\chi_k[i]z_d[i]\right)e_{(x,T)\setminus k},
\end{equation}
where $c(T,k)=\frac{|\vee T|}{|\vee(T\setminus k)|}$ considering $T$ as a set of atoms in the poset $\P(\A)$.
\end{theorem}
\begin{proof}
From Theorem \ref{thm:Leray}, we know that $B(\mathcal{A})$ is isomorphic as a graded algebra to $E_2(\mathcal{A})$, the second page of the Leray spectral sequence. 
Since $E_2^{p,q}$, hence also $E_r^{p,q}$, vanishes unless $q=j(d-1)$ for some $j$, the differential on the $r$th page is trivial unless $d-1$ divides $r-1$. In particular, the first nontrivial differential is on the $d$th page, and $E_d=E_2\cong B(\A)$. 
This differential is determined by Gysin maps
\begin{align*} 
E_2^{0,j(d-1)}\cong \bigoplus_{\rk(x)=j} H^0(x;\Z)\otimes A^j(\P_{\leq x}) 
&\longrightarrow 
\bigoplus_{\rk(y)=j-1} H^d(y;\Z)\otimes A^j(\P_{\leq y}) \cong E_2^{d,(j-1)(d-1)} \\
1\otimes e_{(x,T)} &\mapsto \sum_{(y,T\setminus k)<(x,T)} \sgn(k,T\setminus k) [x]_y\otimes e_{(y,T\setminus k)} 
\end{align*}
and we thus aim to describe the  classes $[x]_y\in H^d(y;\Z)$ explicitly.
Firstly, the class of $H_k=\chi_k^{-1}(g_i)$ in $H^d(\G^r;\Z)$ may be expressed
as
\[ [H_k]_{\G^r} = \chi_k^*(z_1\cdots z_d) = 
\left(\sum_{i=1}^r \chi_k[i]z_1[i]\right)\cdots\left(\sum_{i=1}^r\chi_k[i]z_d[i]\right),
\]
which determines $\delta(e_{(H_k,k)})$.
Now, suppose $(x,T)$ is independent, $k\in T$, and $y\subseteq x$ is a connected component of $\cap_{i\in T\setminus k} H_i$. 
Since the connected components of $\cap_{i\in T}H_i$ correspond to the elements in $\vee T\subseteq\P(\A)$, and similarly for $T\setminus k$, the intersection $y\cap H_k$ has 
$c(T,k)=|\vee T|/|\vee(T\setminus k)|$ connected components, each of which is a translate of $x$.
Then $c(T,k)[x]_y=[y\cap H_k]_y\in H^d(y)$ is equal to the image of $[H_k]_{\G^r}$ through the restriction $H^*(\G^r)\to H^*(y)$,
from which we obtain the formula \eqref{eq:delta}.

We now need $\Q$-coefficients to show that all higher differentials vanish via mixed Hodge theory.
Note that $E_2^{p,q}\cong H^p(x;\Q)\otimes H^{j(d-1)}(\M_x;\Q)$ is pure of weight $p+jd$, since $H^p(x;\Q)$ is pure of weight $p$ by compactness and $H^{j(d-1)}(\M_x;\Q)$ is pure of weight $jd$ by \cite[Prop. 8.6]{CH}. 
Now the differential on the $r$th page, with $r=k(d-1)+1$, maps $E_r^{p,j(d-1)}$ of weight $p+jd$ to $E_r^{p+k(d-1)+1,(j-k)(d-1)}$ of weight $p+jd-k+1$. As the differential must preserve weights, it is necessarily trivial unless $k=1$, that is, $r=d$. Hence, $E_{d+1}=E_\infty$.

Thus, $H^*(B(\A)\otimes\Q,\delta)$ is isomorphic to the associated graded of $H^*(\M(\A);\Q)$ with respect to the Leray filtration. But, since the components $E_\infty^{p,j(d-1)}$ contributing to $H^{p+j(d-1)}(\M(\A);\Q)$ are pure of distinct weights, the Leray filtration coincides with the weight filtration. The associated graded with respect to the weight filtration is isomorphic to $H^*(\M(\A);\Q)$ itself, completing the proof.
\end{proof}

\begin{example}\label{ex:H(B;Q)}
The differential $\delta$ on the algebra $B(\A)$ from \Cref{ex:ellipticB} is determined by
\begin{align*}
\delta(e_1) &= z_1[1]z_2[1], \qquad
\delta(e_2) = z_1[2]z_2[2], \qquad
\delta(e_3) = (2z_1[1]-z_1[2])(2z_2[1])-z_2[2]),\\
\delta(e_{(p_1,12)}) &= z_1[1]z_2[1]e_2-z_1[2]z_2[2]e_1, \\
\delta(e_{(p_1,13)}) &= z_1[1]z_2[1]e_3-(2z_1[1]-z_1[2])(2z_2[1])-z_2[2])e_1 = z_1[1]z_2[1]e_3-z_1[2]z_2[2]e_1,\\
\delta(e_{(p_i,23)}) &= \frac{1}{4}z_1[2]z_2[2]e_3-\frac{1}{4}(2z_1[1]-z_1[2])(2z_2[1])-z_2[2])e_2
= z_1[1]z_2[1]e_3-z_1[1]z_2[1]e_2,
\end{align*}
where $1\leq i\leq 4$.
We can then compute the  Poincar\'e polynomial 
\[ \sum_{n\geq 0} \dim H^n(\M(\A);\Q) t^n = 1+4t+8t^2 \]
with the following elements representing a basis for the cohomology $H^*(\M(\A);\Q)$
\begin{align*}
& 1,\quad z_1[1],\quad z_2[1],\quad z_1[2],\quad z_2[2], \quad
 z_1[1]z_1[2],\quad z_2[1]z_2[2],\quad z_1[1]z_2[2]\\
& 2z_1[2]e_1+z_1[1]e_2-z_1[1]e_3,\quad 2z_2[2]e_1+z_2[1]e_2-z_2[1]e_3 \\
& e_{(p_1,12)}-e_{(p_1,13)}+e_{(p_2,23)},\quad e_{(p_1,12)}-e_{(p_1,13)}+e_{(p_3,23)},\quad e_{(p_1,12)}-e_{(p_1,13)}+e_{(p_4,23)} 
\end{align*}

The formula for the differential $\delta$ in \Cref{thm:cpt} holds even with integer coefficients. 
While we can't apply the other statements in the theorem to integer cohomology, we can and will go through a spectral sequence argument to compute it in this example.
The nonzero terms in $E_3$ are:
\[
E_3^{0,0}\cong\Z
\qquad
E_3^{1,0}\cong\Z^4
\qquad
E_3^{0,2}\cong\Z^3
\qquad
E_3^{1,1}\cong\Z^2
\qquad
E_3^{2,0}\cong\Z^3\oplus\Z_2
\]
The 2-torsion in $E_3^{2,0}$ arises because the image of $\delta\colon B^{0,1}\to B^{2,0}$ is spanned by $z_1[1]z_2[1]$, $z_1[2]z_2[2]$, and $2(z_1[1]z_2[2]+z_1[2]z_2[1])$.
Since $E_3$ is concentrated in only these 5 places, all higher differentials necessarily vanish. Thus, $E_\infty=E_3$.
Since the torsion is concentrated in  $E_\infty^{*,0}\subseteq H^*(M(\A);\Z)$, all extensions are trivial. We conclude
\[
H^0(M(\A);\Z)\cong\Z
\qquad
H^1(M(\A);\Z)\cong\Z^4 
\qquad
H^2(M(\A);\Z)\cong\Z^8\oplus\Z_2
\]

More generally, replacing $\chi_3=(2,-1)$ with $(m,-1)$ in the definition of $\A$ yields $\Z_m\subseteq H^2(M(\A);\Z)$.
\end{example}

\begin{remark}\label{rmk:noncpt}
At the end of the proof of \Cref{thm:cpt}, we were able to recover the algebra structure on $H^*(\M(\A);\Q)$ from the spectral sequence because the Leray filtration coincided with the weight filtration. 
This argument could extend to recover the algebra structure in the noncompact case 
if the terms contributing to $H^n(\M(\A);\Q)$ had pure distinct weights. To see when this works, suppose $\G=(S^1\times S^1)^e\times(\C^\times)^t\times\C^l$ and let $d=\dim\G=2(e+t+l)$.
\begin{itemize}
\item When $e=t=0$ and $l>0$, then $H^q(\M(\A);\Q)\cong  
E_2^{0,q}$ and is pure weight $jd$ where $q=j(d-1)$.
\item When $e>0$, $t=0$, and $l>0$, then the term
$E_2^{n-j(d-1),j(d-1)}$ has pure weight $n+j$.
\item When $e=0$, $t>0$, $l\geq 0$, and $d>2$, then the term
$E_2^{n-j(d-1),j(d-1)}$ has weight $2n+j(-d+2)$.
\end{itemize}
In these three cases, then, one has $B(\A)\otimes\Q\cong H^*(\M(\A);\Q)$, and the Hilbert series formula of \Cref{lem:Bvs} can be refined to encode the mixed Hodge numbers.
The argument fails in other noncompact cases because the terms contributing to $H^n(\M(\A);\Q)$ do not have distinct weights. Although we do not expect an algebra isomorphism in the case $e\geq0$, $t=1$, $l=0$, there is still an algebra isomorphism in the case 
$t+l>1$, as per \Cref{thm:noncpt}.
\end{remark}

\bibliographystyle{alpha}
\bibliography{biblio}
\end{document}